\documentclass[12pt]{amsart}
\usepackage{latexsym,amsfonts,amsmath,amssymb,amsthm,url,amsbsy,amscd}%,yfonts}
\usepackage[english]{babel}
\usepackage[latin1]{inputenc}
\usepackage{graphicx,psfrag,epsfig}
\usepackage{enumerate}
\numberwithin{equation}{section}
\newtheorem{theorem}{Theorem}[section]
\newtheorem{proposition}[theorem]{Proposition}%[section]
\newtheorem{corollary}[theorem]{Corollary}%[section]
\newtheorem{lemma}[theorem]{Lemma}%[section]
\newtheorem{definition}[theorem]{Definition}%[section]
\newtheorem{remark}[theorem]{Remark}%[section]

\newenvironment{proofad}{\removelastskip\par\medskip   % inizio e fine dimostrazione
\noindent{\em Proof of $\text{}\!$ {\rm Theorem  \ref{th:convPM}}.}
\rm}{\penalty-20\null\hfill$\square$\par\medbreak} %%

\newenvironment{proofadmain}{\removelastskip\par\medskip   % inizio e fine dimostrazione
\noindent{\em Proof of $\text{}\!$ {\rm Theorem  \ref{th:main}}.}
\rm}{\penalty-20\null\hfill$\square$\par\medbreak} %%

\usepackage{color}
\definecolor{darkred}{rgb}{0.8,0,0}
\definecolor{darkblue}{rgb}{0,0,0.7}
\definecolor{darkgreen}{rgb}{0,0.4,0}

\newcommand{\eps}{\varepsilon}
\newcommand{\id}{{ I}}
\renewcommand{\d}{{\rm d}}
\newcommand{\dd}{\d}

\newcommand{\R}{{\mathbb R}}
\newcommand{\NN}{{\mathbb N}}
\newcommand{\N}{{\mathbb N}}

\renewcommand{\L}{{\rm L}}
\newcommand{\HH}{{\mathcal H}}
\newcommand{\PP}{{\mathcal P}}

\renewcommand{\AA}{{\mathcal A}}

\newcommand{\C}{{\mathbb C}}
\newcommand{\FF}{{\mathcal F}}
\newcommand{\GG}{{\mathcal G}}

\renewcommand{\SS}{{\mathcal S}}
\newcommand{\KK}{{\mathcal K}}

\newcommand{\VV}{{\mathcal V}}
\renewcommand{\ln}{\log}

\newcommand{\un}{{\rm 1\kern -2.5pt l}}

\newcommand{\de}{\partial}
\newcommand{\supp}{{\rm supp}}

\usepackage[normalem]{ulem}     % "\sout" for striking out in text mode
\newcommand{\RRR}{\color{black}}

\begin{document}
%%%%%%%%%%%%%%%%%%%%%%%%%%%%%%%%%%%%%%%%%%%%%%%%%%%%%%%%%%%%%%%%%%%%%%
%%%%%%%%%%%%%%%%%%%%%%%%%%%%%%%%%%%%%%%%%%%%%%%%%%%%%%%%%%%%%%%%%%%%%%
\title[]{A gradient flow approach to the porous medium equation with fractional pressure}
\date{\today}
%{35K65, 35K40, 47J30, 35Q92, 35B33}
%%%%%%%%%%%%%%%%%%%%%%%%%%%%%%%%%%%%%%%%%%%%%%%%%%%%%%%%%%%%%%%%%%%%%%%%%%%%%%%
%%%%%%%%%%%%%%%%%%%%%%%%%%%%%%%%%%%%%%%%%%%%%%%%%%%%%%%%%%%%%%%%%%%%%%%%%%%%%%%
\begin{abstract}
We consider a family of {\RRR porous media equations with fractional pressure,} recently studied by Caffarelli and V\'azquez.
We show the construction of a weak solution as Wasserstein gradient flow  of a square fractional Sobolev norm.
Energy dissipation inequality,  regularizing effect and decay estimates for the $L^p$ norms are established.
Moreover, we show  that a
classical porous medium equation  can be obtained as a limit case.
\end{abstract}

\author{S. Lisini}
\address{Stefano Lisini -- Universit\`a degli Studi di Pavia, Dipartimento di Matematica "F. Casorati", via Ferrata 1, I-27100 Pavia, Italy}
\email{stefano.lisini@unipv.it}
\author{E. Mainini}
\address{Edoardo Mainini -- Universit\`a degli Studi di Genova, 
Dipartimento di
   Ingegneria meccanica, energe\-tica, gestionale
e dei trasporti (DIME), Piazzale Kennedy 1, I-16129 Genova, Italy
% \newline
%$\&$ Faculty
   %of Mathematics, University of Vienna,
 %Oskar-Morgenstern-Platz 1,
  % A-1090 Vienna, Austria
}
\email{edoardo.mainini@unipv.it}
\author{A. Segatti}
\address{Antonio Segatti -- Universit\`a degli Studi di Pavia, Dipartimento di Matematica "F. Casorati", via Ferrata 1, I-27100 Pavia, Italy}
\email{antonio.segatti@unipv.it}
%\pagestyle{plain}
%
%
%\author{}
%\address{}
%\email{}
%
%\author{}
%\address{}
%\email{}

\thanks{}

\maketitle
%\tableofcontents
%\newpage
\section{Introduction}

{\RRR 
%Let $d\ge 1$ be the dimension of the Euclidean space and $s\in (0,\min\{1,\frac{d}{2} \})$ be a fixed 
%parameter. 
 We consider the evolution  problem }
\begin{equation}\label{equation}
\begin{cases}
\partial_t u-\mathrm{div}(u\nabla v )=0 &\mbox{in }\mathbb{R}^d\times (0,+\infty), \\
(-\Delta)^sv=u &\mbox{in }\mathbb{R}^d\times(0,+\infty),\\
u(0)=u_0,
\end{cases}
\end{equation}
where the initial datum $u_0$ is a Borel probability measure on ${\mathbb{R}^d}$, {\RRR $d\ge1$,   and $0<s<\min\{1,\frac{d}{2} \}$.}
The linear operator $(-\Delta)^s$ is the $s$-fractional Laplacian on $\mathbb{R}^d$, 
defined by means of  Fourier transform 
%$\mathfrak{F} v(\xi) = \hat{v}(\xi): = \int_{\mathbb{R}^d}e^{-ix\cdot\xi}\, v(x)\d x$
as 
%\[
%\Ls v(x) = \mathfrak{F}^{-1}(\vert\xi\vert^{2s}\hat{v}(\xi)).
%\]
$$\widehat{((-\Delta)^sv)}(\xi)=|\xi|^{2s}\hat v(\xi).$$
%\begin{equation}\label{fractionalLaplace}
%(-\Delta)^sv=u \qquad\mbox{in }\mathbb{R}^d.
%\end{equation}
%Our convention for the Fourier transform is $\hat \varphi(\xi)= \int_{\mathbb{R}^d}e^{-ix\cdot\xi}\,  \varphi(x)\,\d x.$
We  define  the Riesz kernel $K_{s}$ by  the relation
$\hat K_{s}(\xi)=|\xi|^{-2s}$, that is,
$$ K_{s}(x)=C_{d,s}{|x|^{-d+2s}},$$
where  $C_{d,s}$ is a normalization constant. With our convention for the Fourier transform, i.e., $\hat \varphi(\xi)= \int_{\mathbb{R}^d}e^{-ix\cdot\xi}\,  \varphi(x)\,\d x,$ we have
\begin{equation}\label{kernelconstant}
C_{d,s}=\pi^{-d/2}2^{-2s}\Gamma(d/2-s)/\Gamma(s),
\end{equation}
where $\Gamma$ is the Euler Gamma function,  see  for instance \cite[Section 1.2.2]{AH}.
The relation between $u$ and $v$, in the
   second equation of \eqref{equation}, is understood as $v=K_{s}\ast u$. Therefore,  problem \eqref{equation}
corresponds to an evolution repulsive interaction equation, characterized by the Riesz kernel $K_{s}$.

%The relation between $u$ and $v$, in the
   %second equation of \eqref{equation}, is understood as $v=K_{s}\ast u$.
%\begin{equation}\label{fractionalLaplace}
%(-\Delta)^sv=u \qquad\mbox{in }\mathbb{R}^d.
%\end{equation}
%Therefore,  problem \eqref{equation}
%corresponds to an evolution repulsive interaction equation, characterized by the Riesz kernel $K_{s}$.

Problem \eqref{equation} has been  studied by Caffarelli and V\'azquez in \cite{CV},
where  existence of solutions  was proved for non-negative bounded initial data 
which decay exponentially fast at infinity.  
The existence result of \cite{CV} has been generalized to $L^1$ positive initial data in \cite{CSV} and to positive finite measure data
{\RRR  in \cite{SV, SDV}}. 
Moreover, \cite{CSV, CV3} contain comprehensive results about H\"older regularity of solutions.  
Barenblatt profiles and asymptotic behavior are investigated in \cite{CV2}.
Exponential convergence towards stationary states in one space dimension, 
after changing to self similar variables, has been obtained in \cite{CHSV}. 
{\RRR More general nonlocal porous media equations are considered for instance in \cite{BIK, SDV, SDV2, SDV3}. See
also \cite{Vcime} and the references therein.}

The system \eqref{equation} is derived 
%(see \cite{CV}) 
by starting from the continuity equation 
\begin{equation*}\label{continuityeq}
\partial_t u + \hbox{div}(u \mathbf{v}) = 0, 
\end{equation*}
governing the evolution of the density distribution $u$, %(of a population, for instance), 
driven by a velocity vector field $\mathbf{v}$. Now, as it happens for the classical porous medium equation, 
we suppose that $\mathbf{v}$ is the gradient of a scalar function $v$, the pressure, which is assumed to
be a function of the density $u$. 
%That is, we assume that $\mathbf{v} =- \nabla \mathfrak{K}[u]$. 
%Different choices of this functional dependence lead to 
%different models. 
The system \eqref{equation} emerges by choosing the nonlocal
closing relation %(as it happens in chemotaxis)
$\mathbf{v} := -\nabla v = -\nabla (K_s\ast u)$. 

 Let us briefly discuss the two extreme cases $s=0$ and $s=1$.
When $s=0$, the second equation formally reduces 
to the identity $v =u$ and thus the system in \eqref{equation} becomes
\begin{equation}\label{s=0}
%\partial_t u - \hbox{div}(u\nabla u) 
\partial_t u - \frac{1}{2}\Delta u^2 = 0,
\end{equation}
that is a classical (local) porous medium equation. 
Among the other results, in this paper  
we will make this transition rigorous (see Theorem \ref{th:convPM}). 
The other extreme situation corresponds to the case $s=1$, {\RRR $d\ge 2$,} where  the second equation 
becomes $-\Delta v = u$. The resulting system \eqref{equation} is related to the 
Chapman-Rubinstein-Schatzman's mean field model in superconductivity (see \cite{CRS}) and to the E's model in superfluidity, 
at least for positive solutions (see \cite{E}). Existence for this system when $s=1$ was first proved in two space dimensions
in \cite{LZ}. More recently, Serfaty and V\'azquez \cite{SV} proved that the solutions of the system 
\eqref{equation} converge in a proper way when $s\nearrow 1$
to the solutions of the corresponding system with $s=1$.

%The physical derivation of the system \eqref{equation} that we have briefly outlined,
 %suggests that
%\eqref{equation} has indeed the structure of a Wasserstein gradient flow, 
%The structure of a continuity equation
%eqref{equation} has the form of a continuity equation 
%\eqref{continuityeq},
%in which the velocity $\mathbf{v}$ is a gradient of a convolution of the density $u$ and given kernel,  
% a variational derivative of a suitable functional of the density $u$, 
 %suggest  an underlying gradient flow structure as specified in the 
%book \cite{AGS}.  

%The goal of this paper is to investigate the gradient flow structure of problem \eqref{equation}.
 \subsection*{The gradient flow structure}
Our main contribution is the rigorous construction
of non-negative solutions for the Cauchy problem  \eqref{equation}
as trajectories of a gradient flow.
More precisely we consider the space $\PP_2(\R^d)$ of Borel probability measures on $\R^d$ 
with finite second moment  endowed with the 2-Wasserstein distance, here denoted by $W$ (see Section \ref{Sec:notation}).
For  $u\in \PP_2(\R^d)$ we define the energy functional
%\begin{equation*}
%\mathcal{F}(u)=\frac12\int_{\mathbb{R}^d}|(-\Delta)^{-s/2}u|^2=\int_{\mathbb{R}^d}|(-\Delta)^{s/2}v|^2=\int_{\mathbb{R}^d}vu=\int_{\mathbb{R}^d\times\mathbb{R}^d} K_s(x-y)u(x)u(y),
%\end{equation*}
\begin{equation*}
	\FF_s(u)=\frac12\|u\|^2_{\dot H^{-s}(\R^d)}:=\frac12\frac1{(2\pi)^d} \int_{\mathbb{R}^d}|\xi|^{-2s}|\hat u (\xi)|^2\,\d\xi,
\end{equation*}
%where the Fourier transform of $u$ is defined by
%\begin{equation*}
%	\hat u(\xi)= \int_{\mathbb{R}^d}e^{-ix\cdot\xi}\,d u(x).
%\end{equation*}
that is,  $\FF_s$ is  the square norm of the   homogeneous
Sobolev space ${\dot H^{-s}(\mathbb{R}^d)}$, see Section \ref{SobolevHomSpace}.
We observe that this functional %that if $u\in L^2(\R^d)$ and $K_s*u\in L^2(\R^d)$, then  Plancherel formula
  admits the alternative representation
$$\FF_s(u)=\frac12\int_{\R^d}\int_{\R^d} K_{s}(x-y)\,\d u(x)\, \d u(y),$$
 enlightening the structure of   an interaction energy, characterized by the Riesz convolution kernel $K_s$. 
 %The pressure $K_s\ast u$ is the formal variation of functional $\FF_s$ with respect to $u$.
%Therefore, the domain of the functional $\FF_s$ coincides with $\dot H^{-s}(\mathbb{R}^d)\cap\PP_2(\mathbb{R}^d)$.
Within the gradient flow interpretation, we prove that
a solution to the Cauchy problem  \eqref{equation} can be obtained by means of the minimizing movement approximation scheme,
applied to the functional $\FF_s$ in the metric space $(\PP_2(\R^d),W)$.
%The minimizing movements approach allows the generalization to metric spaces of  the variational form of the Euler implicit discretization for gradient flows.
A general theory of  minimizing movements in metric spaces and its applications to the  space $(\PP_2(\R^d),W)$ is contained in the book of 
Ambrosio, Gigli and Savar\'e \cite{AGS}. 
The  gradient flow approach in $(\PP_2(\R^d),W)$  was first  exploited by Jordan-Kinderlehrer-Otto in the seminal paper \cite{JKO}.
Let us illustrate the strategy in our case: given $u_0\in\dot H^{-s}(\mathbb{R}^d)\cap\PP_2(\mathbb{R}^d)$ and $\tau>0$
{\RRR we introduce the following time discretization scheme. We consider 
a uniform partition of size $\tau$ of the time interval $[0,+\infty)$ and we let}
$u_\tau^0$ be a suitable approximation of the initial datum (see \eqref{utau0}). Then, we recursively define
\begin{equation}\label{minmov1}
u_\tau^k\in{\rm Argmin}_{u\in\PP_2(\mathbb{R}^d)}\left\{ \FF_s(u)+\frac1{2\tau}\,W^2(u,u^{k-1}_\tau)\right\}, \qquad \mbox{for } k=1,2,\ldots.
\end{equation}
%The existence and uniqueness of solution for the minimization problem in \eqref{minmov1} will be established in Section \ref{Sec:notation}. 
If $\{u_\tau^k\}_{k\in\mathbb{N}}\subset\PP_2(\mathbb{R}^d)$ is a sequence defined by \eqref{minmov1}, we introduce the piecewise constant interpolation
$$
u_\tau(t):=u^{\lceil t/\tau\rceil}_{{\RRR\tau}},\qquad t\in [0,+\infty),
$$
where $\lceil a\rceil:=\min\{m\in \mathbb{N}:m>a\}$ is the upper integer part of the real number $a$. {\RRR We refer to $u_\tau$ as discrete solution.}
We prove that this family of piecewise constant curves admits limit points as $\tau\to 0$,
and that a limit curve is a {\RRR weak solution to \eqref{equation}, satisfying some additional properties (see Theorem \ref{th:main}).}

%When $s=1$ and the system is posed in a bounded domain,
%this gradient flow interpretation has been analyzed by
%Ambrosio, Serfaty and the second author of this paper in \cite{AMS, AS}. 
%More generally, 
Nonlocal evolution equations with  singular kernels appear in several mathematical models. 
However, up to now the corresponding gradient flow approach is limited to less singular interactions. 
Besides the works \cite{AMS, AS}, dealing with the Chapman-Rubinstein-Schatzman superconductivity model,
%The closest problems in the literature feature the standard Laplacian in the second equation of \eqref{equation}, 
%in place of the fractional Laplacian.  
%An instance of this kind, in bounded domains,  comes from  superconductivity models, and it is analyzed in \cite{AMS, AS}.
gradient flows of equations involving Newtonian interaction   appear in the  study of  the %parabolic-ellitic
 Keller-Segel model 
for chemotaxis, see \cite{BCC} and the reviews \cite{B,B2}.  
The approach we propose here  is strictly related to the latter contributions, and problem \eqref{equation}, 
with the corresponding functional $\FF_s$, turns out to be a remarkable example of Wasserstein gradient flow.

\subsection*{The main result}
We shall now state the results.  The main one  is the following Theorem \ref{th:main}, which contains all the properties of the gradient flow solutions.
Throughout the paper we denote by $\HH:\PP_2(\R^d)\to (-\infty,+\infty]$
the entropy defined by $\HH(u):=\int_{\R^d}u\log u \,\dd x$ if $u$ is absolutely continuous with respect to the Lebesgue measure and
$\HH(u)=+\infty$ otherwise.
We use the notation $D(\HH)=\{u\in\PP_2(\R^d):\HH(u)<+\infty\}$ for the domain of $\HH$. Moreover, in the statement of Theorem \ref{th:main}, the approximation  $u_\tau^0$ of the initial datum $u_0$ is not arbitrary, but  given by the suitable Gaussian regularization   defined  in  Section \ref{section:functional} below, see \eqref{utau0}. 
{\RRR See also Section \ref{Sub:Wasserstein} for the definition of narrow convergence and Section \ref{subsection:MM} for the definition of the space $ AC^2([0,+\infty);(\PP_2(\mathbb{R}^d),W))$.}

\begin{theorem}\label{th:main}
Let {\RRR $d\ge 1$, $0<s<\min\{1,\frac{d}{2}\}$}  and $u_0\in \dot H^{-s}(\mathbb{R}^d)\cap\PP_2(\mathbb{R}^d)$.
Then the following assertions hold:
\begin{itemize}
\item[i)] {\bf Existence and uniqueness of discrete solutions.} For every $\tau>0$, after having defined $u_\tau^0$ 	by \eqref{utau0}, there exists a unique sequence
$\{u_\tau^k:k=1,2,\ldots\}$ satisfying \eqref{minmov1}. 
\item[ii)] {\bf Convergence and regularity.} For every vanishing sequence $\tau_n$ there exists a (not relabeled) subsequence  $\tau_{n}$ and a curve
$u\in AC^2([0,+\infty);(\PP_2(\mathbb{R}^d),W))$ such that
\[
 u_{\tau_{n}}(t) \to u(t)\quad \mbox{narrowly as $n\to\infty$, for any $t\in[0,+\infty)$.}
\]
{\RRR Moreover, } $u\in L^2((T_0,T); H^{1-s}(\R^d))$ for every 
{\RRR $0<T_0<T$},
 and
\[
 u_{\tau_{n}}\to u \quad \mbox{strongly in } L^2((T_0,T);L^2_{loc}({\R}^d)) \quad \mbox{ as $n\to\infty$%, \quad $\forall T_0,T$ such that $T>T_0>0$
 .}
\]
Defining $v_\tau(t):=K_{s}*u_\tau(t)$ and $v(t):=K_{s}*u(t)$ $\forall\,t>0$,  we have that $\nabla v \in L^2((T_0,T);L^2(\mathbb{R}^d))$
for every {\RRR $0<T_0<T$}, and
\[
 \nabla v_{\tau_{n}}\to \nabla v \quad \mbox{weakly in } L^2((T_0,T);L^2({\R}^d)) \quad \mbox{ as $n\to\infty$%, \quad $\forall T_0,T$ such that $T>T_0>0$
 .}
\]
\item[iii)] {\bf Solution of the equation.} Given $u,v$ from point {\rm ii)}, the first equation in \eqref{equation} is satisfied in the following {\RRR weak form:}
\[
\int_0^{+\infty}\int_{\R^d} ( \partial_t\varphi - \nabla\varphi\cdot \nabla v )u\, \dd x \, \dd t=0, \quad
\text{for all }\varphi\in C^\infty_c((0,+\infty)\times\R^d).
\]
\item[iv)] {\bf Energy dissipation inequality.} Given $u,v$ from point {\rm ii)}, there holds
\begin{equation}\label{EDI}
\FF_s(u(t))+\int_0^t\int_{\mathbb{R}^d}|\nabla v(r)|^2u(r)\,\d x\,\d r \le \FF_s(u_0), \qquad \forall \;t\in [0,+\infty).
\end{equation}
\item[v)] {\bf Regularizing effect and decay estimates.}
 For every $p\in[1,+\infty]$ there is a constant $C_p$ depending only on $p,d$ and $s$  (independent of $u_0$) such that
\begin{equation*}
	\|u(t)\|_{L^p(\R^d)}\leq C_p t^{-\gamma_p} \qquad \forall t>0,
\end{equation*}
where $\gamma_p=\frac{p-1}{p}\frac{d}{d+2(1-s)}$ for $p<+\infty$ and $\gamma_\infty=\frac{d}{d+2(1-s)}$.
In particular $u(t)\in  D(\HH)\cap L^p(\R^d)$ for every $t>0$.
\item[vi)] {\bf Entropy estimates.}
If, in addition, $u_0\in D(\HH)$, then
\[
	\HH(u(t))\leq \HH(u_0), \qquad \forall t>0.
\]
If $u_0\in L^p(\R^d)$ for some $p\in [1,+\infty]$, then
\[
	\|u(t)\|_{L^p(\R^d)}\leq \|u_0\|_{L^p(\R^d)}, \qquad \forall t>0.
\]
 %Moreover, the results of point {\upshape ii)} also hold for $T_0=0$.
%There exists a constant $C_0$ depending only on $d$ and $s$ such that
%\[
%	\KK(u(t))\leq C_0 t^{-\gamma_0} \qquad \forall t>0,
%\]
%where $\gamma_0=\frac{d}{2d+4(1-s)}$.\\
\end{itemize}
\end{theorem}

\begin{remark}\label{firstremark}\rm
The proof of Theorem \ref{th:main} will be given as the collection of different results through the paper. Let us give {\RRR some}
comments here.
\begin{itemize}
\item If $u_0\in D(\HH)$, then the results of point {\upshape ii)} also hold for $T_0=0$ and the results of points 
{\upshape i)}-{\upshape ii)}-{\upshape iii)}-{\upshape iv)}  do not require the approximation {\RRR of} the initial datum 
(that is, we could define $u_\tau^0=u_0$ in this case).
\item The value of the constant $C_p$ in point {\upshape v)} is explicit, see Lemma \ref{lemma:decay2} below for $p\in(1,+\infty)$ and 
Theorem \ref{th:infinity} for $p=+\infty$. 
{\RRR
If $p=1$  we have $C_1=1$ and equality holds in points  {\upshape v)} and  {\upshape vi)} 
because mass conservation is an automatic consequence of the Wasserstein gradient flow construction of solutions.}
\item For every $p\in [1,+\infty]$ the exponent $\gamma_p$ in point {\upshape v)} is sharp, since the Barenblatt-type solutions constructed in  \cite{CV2}
have the same decay rate.
\item The solutions that we construct are weak energy solutions in the terminology of Caffarelli and Vazquez.
{\RRR Consequently they are also H\"older continuous thanks to \cite[Theorem 5.1]{CSV}. The finite speed of propagation
is obtained by Caffarelly and Vazquez in \cite{CV} and relies on their construction of weak solutions (see also \cite{I} and \cite{SDV2}).
It would be an interesting
problem to obtain the finite speed of propagation directly from our discrete scheme.}
\RRR \item Theorem \ref{th:main} holds if we consider  positive measure data in $\dot H^{-s}(\mathbb{R}^d)$, with finite second moment and mass $M>0$. 
%In such case the constant $C_p$ in point v) has to be multiplied by $M^{\ell_p}$, where $\ell_p$ is given in Remark \ref{firstremark}. positive measure data with mass $M>0$. 
In such case, the constant $C_p$ from point {\upshape v)} gets multiplied by $M^{\ell_p}$ where $\ell_p=\tfrac{2p(1-s)+d}{2p(1-s)+dp}$ if $p\in[1,+\infty)$ and $\ell_\infty=\tfrac{2(1-s)}{2(1-s)+d}$. This scaling is the same  obtained in \cite{CSV} for positive $L^1(\mathbb{R}^d)$ data. See also Remark \ref{finalremark} below.
\end{itemize}
\end{remark}

Let us summarize the main techniques and the strategy that we shall use in the paper.
We start with the analysis of the discrete variational problem \eqref{minmov1} proving existence and uniqueness of the {\RRR discrete solutions}.
Moreover we analyze the regularity of minimizers, which are indeed shown to belong to $\dot H^{1-s}(\mathbb{R}^d)$, 
and not only to $\dot H^{-s}(\mathbb{R}^d)$. In order to do this we make use of the  flow interchange technique, 
described by McCann, Matthes and Savar\'e in \cite{MMS}. 
The improved regularity of minimizers allows to perform variations along transport maps and to derive a
corresponding Euler-Lagrange equation, which yields a discrete formulation of problem \eqref{equation}.
Moreover, the obtained regularity estimates entail sufficient compactness in order 
to pass to the limit in such discrete formulation, obtaining a weak solution to problem \eqref{equation}.  
 Finally, in order to obtain the energy dissipation inequality of functional $\FF_s$ along the solution we use the De Giorgi variational interpolation.
In these steps we often work in Fourier variables. This  reveals useful and appears quite natural, starting from the definition of the energy functional.   

The other important features that we discuss are the regularizing effect and the decay rate at infinity of $L^p$ norms 
stated in point {\upshape v)} of Theorem \ref{th:main}. 
We stress that the regularizing effect allows to treat the case of general $\PP_2\cap\dot H^{-s}$ initial data.
The decay rate of the $L^p$ norms was already obtained in \cite{CSV}.
From our point of view, this relates to the interesting issue of finding general  $L^p$ estimates 
at the discrete level of the  minimizing movements scheme, 
along with the corresponding decay rates for large times, which is new in this framework. 
At the discrete level, for $p<+\infty$, we obtain an estimate of the form
\begin{equation*}
    \|u_\tau^k\|_{L^p(\mathbb{R}^d)} \leq \min \{ \|u_\tau^0\|_{L^p(\mathbb{R}^d)},C_p(k\tau)^{-\gamma_p}\}  +R_\tau
    %\frac{C_p}{\sqrt{2}}\tau\|u_\tau^0\|^{p\beta}_{L^p}  
     ,\qquad k=1,2,3,...,
\end{equation*}
where $\gamma_p=\frac{p-1}{p}\frac{d}{d+2(1-s)}$ and $R_\tau$ is a suitable remainder term.
Such an estimate is proved by combining the flow interchange technique 
% making variations along solutions of gradient flows of auxiliary geodetically convex entropies.
with  Sobolev inequalities. 
The  term $R_\tau$ is then shown to vanish as $\tau\to0$, thus yielding the desired decay estimates of the $L^p$ norms for $p<+\infty$. 
However, it is not possible to directly pass to the limit as $p\to +\infty$, because the multiplicative constant $C_p$ blows up.
We note that an analogous difficulty for the case of the porous medium equation
was observed for instance in \cite{BG}, when trying to obtain the decay rate of the $L^\infty$ norm 
by making use of Sobolev inequalities.

 In order to obtain the $L^\infty$ decay, a refined argument is indeed necessary. 
Here, we adapt the techniques of Caffarelli-Soria-V\'azquez \cite{CSV} to the discrete setting. 
Their approach  for proving $L^\infty$ decay estimates was previously introduced by Caffarelli and Vasseur \cite{CVas,CVas2} 
for the case of the semigeostrophic equation, and it is based on the De Giorgi technique for elliptic equations. 
In order to  apply this technique within the discrete setting we introduce a sequence of minimizing movements approximations on a smaller scale. 
{\RRR This construction represents one of the main novelties of the paper (see Section \ref{sec:infty}).
The new approximation provides the required informations on the solution,
allowing for an $L^2$ to $L^\infty$ argument to get $L^\infty$ decay with the expected rate $\gamma_\infty=\lim_{p\to+\infty}\gamma_p$, corresponding to the one obtained in \cite{BIK,CSV}.}

\subsection*{The limit as $s\to 0$}
A final result that we prove is the convergence of the constructed solutions to a solution of the 
standard porous medium equation \eqref{s=0} as the fractional  parameter $s$ goes to zero. 
This complements the result of Serfaty and V\'azquez \cite{SV}, where the limiting case as $s\to1$ 
(corresponding to the interaction with the Newtonian potential) is analyzed. 
%On the other hand, uniqueness of solution of \eqref{equation} is still a major open problem.
More precisely, the result is stated in the following Theorem.

\begin{theorem}\label{th:convPM}
Let $u_0\in L^2(\R^d)$ and $\{u_0^s\}_{s\in(0,1)}$ be a family of initial data such that $u_0^s\in D(\FF_s)$, $u_0^s$ converges narrowly to $u_0$ as
$s\to 0$,  $\sup_{s\in(0,1)} \int_{\R^d} |x|^2 \, \d u_0^s(x) <+\infty$  and 
 $\lim_{s\to 0}\FF_s(u_0^s)=\FF_0(u_0)$ where $\FF_0(\cdot) :=\frac 12\|\cdot \|_{L^2(\mathbb{R}^d)}$.
For each $s\in(0,1)$, let $u^s$ be a solution to the corresponding equation  \eqref{equation}, with initial datum $u_0^s$, given by {\rm Theorem \ref{th:main}}.
Let moreover  $u$ be the unique solution of the Cauchy problem for the porous medium equation
\begin{equation}\label{pmequation}
\begin{cases}
\de_t u-\frac12\Delta u^2=0 &\mbox{in }\mathbb{R}^d\times (0,+\infty), \\
u(0)=u_0,
\end{cases}
\end{equation}
satisfying the energy identity
\[
\FF_0(u(T))+\int_0^T\int_{\mathbb{R}^d}|\nabla u(t)|^2u(t)\,\d x\,\d t = \FF_0(u_0), \qquad \forall \;T>0.
\]
Then we  have
\[
u^s(t)\to u(t) \quad\mbox{narrowly  as $s\to 0$ for every $t\geq 0$,}
\]
and, for every $T_0$ and $T$ such that $T>T_0>0$,
\[
u^s\to u \quad \mbox{ strongly in $L^2((T_0,T);L^2_{loc}(\mathbb{R}^d))$ \quad as $s\to 0$,}
\]
\[
 \nabla u^s \to \nabla u \quad \mbox{weakly in } L^2((T_0,T);L^2({\R}^d)) \quad \mbox{ as $s\to 0$.}
\]
\end{theorem}

\vspace{3mm}
\noindent\textbf{Plan of the paper}.
Section 2 introduces the basic framework for gradient flows in the Wasserstein space and for fractional Sobolev norms. 
Section 3 shows the convergence of the scheme to some absolutely continuous curve in $\PP_2(\mathbb{R}^d)$, owing only to the general theory of minimizing movements, and not relating to the specific choice of functional $\FF_s$.
Section 4 introduces the flow interchange, which will be repeatedly used in order to obtain further regularity of minimizers,
the regularizing effect of the dynamics, and the $L^p$  decay estimates for $p\in(1,\infty)$. 
Section 5 is devoted to the Euler-Lagrange equation for discrete minimizers, thus building up the key element for the  existence result.
Section 6 proves existence, by showing that the limit curve found in Step 3 is in fact regular enough for giving sense 
to the term $u\nabla v$ and satisfies equation \eqref{equation}. 
This is moreover a gradient flow solution, so that \eqref{equation} holds in the sense of distributions and 
an energy dissipation inequality for functional $\FF_s$ holds.
Section 7 introduces the double scale approximation and proves the $L^\infty$ decay estimates, 
thus completing the proof of Theorem \ref{th:main}
Eventually, Section 8 contains the proof of Theorem \ref{th:convPM}.

\section{Notation and preliminary results} \label{Sec:notation}
\subsection{Wasserstein distance}\label{Sub:Wasserstein}
We denote by $\PP(\mathbb{R}^d)$ the set of Borel probability measures on $\mathbb{R}^d$.
The narrow convergence in $\PP(\mathbb{R}^d)$ is defined in duality with continuous and bounded functions on $\mathbb{R}^d$, i.e.,
a sequence $\{u_n\}\subset \PP(\R^d)$ narrowly converges to $u\in  \PP(\R^d)$ if
$\int_{\R^d} \phi \, \d u_n \to \int_{\R^d} \phi \, \d u$ for every $\phi\in C_b(\R^d)$,
where  $C_b(\R^d)$ is the set of continuous and bounded functions defined on $\R^d$.

We define $\PP_2(\R^d):=\{u\in\PP(\R^d):  \int_{\R^d} |x|^2 \, \d u(x)<+\infty \}$
the set of Borel probability measure with finite second moment.
The Wasserstein distance {\RRR W} in $\PP_2(\R^d)$ is defined as
\begin{equation}\label{Kanto}
W(u,v):=\min_{\gamma\in\PP({\R^d\times\R^d})}\left\{\left(\int_{\R^d\times\R^d}\!\!\!\!\!|x-y|^2\,\d\gamma(x,y)\right)^{1/2}:
\,(\pi_1)_\#\gamma=u,\,(\pi_2)_\#\gamma=v
\right\}
\end{equation}
where $\pi_i$, $i=1,2$, denote the canonical projections on the first and second factor respectively.
Denoting by $\id$ the identity map in $\R^d$, when $u$ is absolutely continuous with respect to the Lebesgue measure, 
the minimum problem~\eqref{Kanto}
has a unique solution $\gamma$ induced by a transport map $T_u^v$ in the following way:
$\gamma=(\id,T_u^v)_\#u$.
In particular, $T_u^v$ is the unique solution of the Monge optimal transport problem
$$
\min_{S:\R^d\to\R^d}\left\{\int_{\R^d}|S(x)-x|^2\d u(x):\ S_\#u=v\right\}.
$$
%of which~\eqref{Kanto} is the Kantorovich relaxed version.
 Finally, we recall that if also
$v$ is absolutely continuous with respect to Lebesgue measure, then
\begin{equation*}\label{ot}
T_v^u\circ T_u^v=\id\quad\hbox{\rm $u$-a.e.}
\quad\hbox{\rm and}\quad
T_u^v\circ T_v^u=\id\quad\hbox{\rm $v$-a.e.}
\end{equation*}

The function $W:\PP_2(\R^d)\times\PP_2(\R^d)\to\R$ is a distance and the metric space $(\PP_2(\R^d),W)$ is complete and separable.
Moreover the distance $W$ is sequentially lower semi continuous with respect to the narrow convergence, i.e.,
$$ u_n\to u,\quad v_n\to v, \mbox{ narrowly } \Longrightarrow \liminf_{n\to+\infty}W(u_n,v_n)\geq W(u,v),$$
and bounded sets in $(\PP_2(\R^d),W)$ are narrowly sequentially relatively compact.

\subsection{Fourier transform and fractional Sobolev spaces}\label{SobolevHomSpace}
We denote by $\SS(\R^d)$ the Schwartz space of smooth functions with rapid decay at infinity and by $\SS'(\R^d)$ the dual space of tempered distributions.
The Fourier transform of $u\in \SS(\R^d)$ is defined by 
$ \hat u(\xi):= \int_{\mathbb{R}^d}e^{-ix\cdot\xi}u(x)\,\d x. $
The Fourier transform is an automorphism of $\SS(\R^d)$ and by transposition it can be defined on $\SS'(\R^d)$.
Moreover the  Plancherel formula holds
$$ \int_{\R^d} \hat u(\xi)\overline{\hat w(\xi)}\,\d\xi =  (2\pi)^d\int_{\R^d} u(x) w(x)\,\d x,\qquad \forall u,w\in L^2(\R^d).$$

Let $r\in\R$.
For every tempered distribution $u\in\SS'(\R^d)$ such that $\hat u \in L^1_{loc}(\R^d)$,
we define
\[
\|u \|^2_{H^{r}(\R^d)} := \frac1{(2\pi)^d}\int_{\mathbb{R}^d}(1+|\xi|^2)^{r}|\hat u (\xi)|^2\,\d\xi
\]
and
\[
\|u \|^2_{\dot H^{r}(\R^d)} := \frac1{(2\pi)^d}\int_{\mathbb{R}^d}|\xi|^{2r}|\hat u (\xi)|^2\,\d\xi.
\]
The fractional Sobolev space $H^{r}(\R^d)$ is defined by
$$H^{r}(\R^d):=\{u\in\SS'(\R^d): \hat u \in L^1_{loc}(\R^d) ,\; \|u \|_{H^{r}(\R^d)}<+\infty \},$$
and the homogenous fractional Sobolev space  $\dot H^{r}(\R^d)$ is defined by 
$$\dot H^{r}(\R^d):=\{u\in\SS'(\R^d): \hat u \in L^1_{loc}(\R^d) ,\; \|u \|_{\dot H^{r}(\R^d)}<+\infty \}.$$
The next proposition summarizes some basic facts about fractional Sobolev spaces, which will be used many times in the sequel. 
We refer  for instance to \cite[Sections 1.3, 1.4]{BCD}.

\begin{proposition}\label{prop:Sobolev}
The following assertions hold.
\begin{itemize}

\item 
Interpolation. If $r_0<r_1<r_2$ then
$$ \|u \|_{H^{r_1}(\R^d)}\leq \|u \|^{1-\theta}_{H^{r_0}(\R^d)} \|u \|^\theta_{H^{r_2}(\R^d)}\quad\text{and}
\quad \|u \|_{\dot H^{r_1}(\R^d)}\leq \|u \|^{1-\theta}_{\dot H^{r_0}(\R^d)} \|u \|^\theta_{\dot H^{r_2}(\R^d)}, $$
where $\theta$ is defined by $r_1=(1-\theta)r_0+\theta r_2$.
\item
If $r_1<r_2$ then $\|u \|_{H^{r_1}(\R^d)} \leq \|u \|_{H^{r_2}(\R^d)}$.
If $r>0$ then  $\|u \|_{\dot H^{r}(\R^d)} \leq \|u \|_{H^{r}(\R^d)}$.
If $r<0$ then  $\|u \|_{H^{r}(\R^d)} \leq \|u \|_{\dot H^{r}(\R^d)}$.
If $r=0$ then  $\|u \|_{\dot H^{0}(\R^d)} = \|u \|_{H^{0}(\R^d)}= \|u \|_{L^{2}(\R^d)}$.
\item
If $\phi\in \SS(\R^d)$ and $u\in H^r(\R^d)$ then there exists a constant $c$, depending only on $\phi$, $r$ and $d$, such that
$$\|\phi \, u \|_{H^{r}(\R^d)} \leq c\|u \|_{H^{r}(\R^d)}.$$
\item
If $\phi\in\SS(\R^d)$,  $r_1<r_2$ and $\sup_{n\in\N}\|u_n \|_{H^{r_2}(\R^d)}<+\infty$, then
$\{\phi \, u_n:n\in\N  \}$ is relatively compact in $H^{r_1}(\R^d)$.
\end{itemize}
\end{proposition}

Let ${\RRR d\ge 1}$ and $r\in (0,d/2)$. Then  the fractional Sobolev inequality  holds
\begin{equation}\label{fractionalembedding}
	\|u\|_{L^{q}(\mathbb{R}^d)}\le S_{d,r}\|u\|_{\dot H^{r}(\mathbb{R}^d)} %\le  C\|u\|_{H^{r}(\mathbb{R}^d)}
\end{equation}
for any $u\in \dot H^{r}(\mathbb{R}^d)$, where $q:=\frac{2d}{d-2r} >2$ and, see for instance \cite{CT},
\begin{equation} \label{best}
S_{d,r}=2^{-2r}\pi^{-r}\frac{\Gamma(d/2-r)}{\Gamma(d/2+r)}\left(\frac{\Gamma(d)}{\Gamma(d/2)}\right)^{2r/d}.
\end{equation}
 From  \eqref{fractionalembedding} and interpolation {\RRR of $L^p$ norms} we obtain that
for  $q_1,q_2$ such that $1\le q_1< q_2<q=\frac{2d}{d-2r}$, the inequality
\begin{equation*}\label{FractionalGNS1}
	\|u\|_{L^{q_2}(\mathbb{R}^d)}\le S_{d,r}^\theta\|u\|^{1-\theta}_{L^{q_1}(\mathbb{R}^d)}\|u\|^\theta_{\dot H^{r}(\mathbb{R}^d)} %\le  C\|u\|_{H^{r}(\mathbb{R}^d)}
\end{equation*}
holds for any $u\in \dot H^{r}(\mathbb{R}^d)\cap L^{q_1}(\R^d)$, where $\theta=\frac{(q_1-q_2)q}{(q_1-q)q_2}$.
In particular,   for  any $u\in \dot H^{r}(\mathbb{R}^d)\cap L^1(\R^d)$ and  ${\RRR q_2}=2+\frac{2r}{d} $ there holds
\begin{equation}\label{FractionalGNS}
	\|u\|_{L^{q_2}(\mathbb{R}^d)}^{\RRR q_2}\le S_{d,r}^2\|u\|^{2r/d}_{L^{1}(\mathbb{R}^d)}\|u\|^2_{\dot H^{r}(\mathbb{R}^d)}. %\le  C\|u\|_{H^{r}(\mathbb{R}^d)}
\end{equation}
{\RRR
Similarly, from \eqref{fractionalembedding} and the interpolation of $L^p$ norms between the exponents $1<p<\frac{d(p+1)}{d-2r}$, for $p\in(1,+\infty) $ and nonnegative $u\in L^1(\mathbb{R}^d)$ such that $u^{(p+1)/2}\in\dot H^r(\mathbb{R}^d)$ we have
\begin{equation}\label{GNS}
\|u\|_{L^p(\mathbb{R}^d)}^{p+1}\le S_{d,r}^{2\theta}\|u\|_{L^1(\mathbb{R}^d)}^{(1-\theta)(p+1)}\|u^{(p+1)/2}\|_{\dot H^r(\mathbb{R}^d)}^{2\theta},
\end{equation} 
where $\theta=\frac{d(p^2-1)}{p(2r+dp)}$.}

{\RRR
%Since in Sections \ref{Sec} and {Sec} 
In dimension $d=1$, for  $s\in(0,1/2)$, we shall also need the following  inequalities:
%inequalities like \eqref{FractionalGNS} and  \eqref{GNS} for $r>1/2$.
\begin{equation}
\label{eq:Bs}
\|u\|_{L^{4-2s}(\mathbb{R})}^{4-2s}\le S_{1,\frac{1-s}{4-2s}}^{2-2s}\|u\|_{L^1(\mathbb{R})}^{2-2s}\|u\|^2_{\dot H^{1-s}(\mathbb{R})},
\end{equation}
%which holds for any $u\in L^1(\mathbb{R})\cap \dot H^{1-s}(\mathbb{R})$.
and
%whereas similar interpolation arguments and \eqref{GNS} entail
\begin{equation}\label{extrasobolev1}
\|u\|_{L^p(\mathbb{R})}^{p\beta_p}\le S_{1,\frac{1-s}{4-2s}}^{8-4s}\|u\|_{L^1(\mathbb{R})}^{(2p-2sp+1)/(p-1)}\|u^{(p+1)/2}\|^2_{\dot H^{1-s}(\mathbb{R})},
\end{equation}
%holds for any $p\in(1,+\infty)$ and any nonnegative $u\in L^1(\mathbb{R})$ such that $u^{(p+1)/2}\in \dot H^{1-s}(\mathbb{R}^d)$, 
 where $\beta_p=\frac{2(1-s)+p}{p-1}$ and $p\in (1,+\infty)$. Indeed, by \eqref{fractionalembedding} and the interpolation property of Proposition \ref{prop:Sobolev} we have
 \begin{equation}
\label{eq:r_interpol}
\|u\|_{L^{\frac2{1-2r}}(\mathbb{R})}\le S_{1,r}\|u\|_{\dot H^r(\mathbb{R})}\le S_{1,r} \|u\|_{L^2(\mathbb{R})}^{(1-s-r)/(1-s)}\|u\|^{r/(1-s)}_{\dot H^{1-s}(\mathbb{R})}
\end{equation}  for every $r,s\in(0,1/2)$.
Choosing $r=\frac{1-s}{4-2s}$ in \eqref{eq:r_interpol} and interpolating the $L^2$ norm between $L^1$ and $L^{\frac2{1-2r}}$ we obtain \eqref{eq:Bs},
whereas similar interpolation arguments and \eqref{GNS} entail \eqref{extrasobolev1}.
}

If  ${\RRR d\ge 1}$ and {\RRR$ r\in (0,1)$, the scalar product in the  space  $\dot H^r(\mathbb{R}^d)$,} defined by 
\[
\langle v,w\rangle_{r}:= \frac1{(2\pi)^d}\int_{\mathbb{R}^d}|\xi|^{2r}\hat v(\xi)\overline{\hat w(\xi)}\,\d\xi,
\]
 can also be expressed as 
\begin{equation}\label{bilinearform}
\langle v,w\rangle_{r} =\bar C_{d,r}\int_{\R^d}\int_{\R^d} (v(x)-v(y))(w(x)-w(y))|x-y|^{-d-2r}\,\d x\,\d y.
\end{equation}
This equivalence follows from \cite[Proposition 1.37]{BCD}. The value of the positive constant  $\bar C_{d,r}$ can be obtained through the following formal computation. 
Since the Riesz kernel satisfies $\Delta K_r=-K_{r-1}$, using  Plancherel formula and integration by parts we have
\[
\begin{aligned}
\langle v,w \rangle_r&=\frac{1}{(2\pi)^d}\int_{\mathbb{R}^d}|\xi|^{-2(1-r)}|\xi|^2 \hat v(\xi)\overline{\hat w(\xi)}\,\d\xi
%=\frac{1}{(2\pi)^d}\sum_{j=1}^d\int_{\mathbb{R}^d}
%|\xi|^{-2(1-r)}\widehat{\partial_{x_j}v}(\xi)\overline{\widehat{\partial_{x_j} w}(\xi)}\,\d\xi\\
=\int_{\mathbb{R}^d}(K_{1-r}\ast \nabla v)(x)\cdot \nabla w(x)\,\d x\\&
={\RRR\frac12} \int_{\mathbb{R}^d}\int_{\mathbb{R}^d}(\Delta K_{1-r}(x-y))\,(v(x)-v(y))\,(w(x)-w(y))\,\d x\d y\\
&=-{\RRR\frac12}C_{d,-r} \int_{\mathbb{R}^d}\int_{\mathbb{R}^d}|x-y|^{-2r-d}(v(x)-v(y))(w(x)-w(y))\,\d x\d y,
\end{aligned}
\]
thus \eqref{bilinearform} holds with $\bar C_{d,r}=-{\RRR\frac12}C_{d,-r}$,
where $C_{d,-r}<0$ is given by  extending formula \eqref{kernelconstant} to values  of the second index in $(-1,0)$. 
%-\pi^{-d/2}2^{2r}\Gamma(d/2+r)/\Gamma(-r)$. 

We also have the following
%\subsection{Bilinear forms}
%For $v,w\in \dot H^{1-s}(\R^d)$ we define 
%$$ \BB(v,w):= C_{d,s}C_0\int_{\R^d}\int_{\R^d} (v(x)-v(y))(w(x)-w(y))|x-y|^{-d-2(1-s)}\,\d x\,\d y,$$
%where $C_{d,s}$ is defined in \eqref{kernelconstant} and $C_0=2(1-s)(d-2s)$.
%
%If $v,w \in \dot H^{1}(\R^d)$, taking into account that $C_0|x|^{-d-2(1-s)}=\Delta |x|^{-d+2s}$, it can be proved that
%$$ \BB(v,w)= C_{d,s}\int_{\R^d}\int_{\R^d} \nabla v(x)\cdot \nabla w(y)|x-y|^{-d+2s}\,\d x\,\d y.$$

\begin{proposition}\label{prop:BB}
Let ${\RRR d\ge 1}$ and ${\RRR r\in(0,1)}$.
Let $v\in\dot H^r(\mathbb{R}^d)$. If $F:\R\to\R$ is nondecreasing, then $\langle v,F(v)\rangle_r\geq 0$.
If, in addition, $F$ is Lipschitz continuous on $\mathbb{R}^d$, then $F\circ v\in \dot H^{r}(\mathbb{R}^d)$ and there hold $$\langle v,F(v)\rangle_r\leq L\langle v,v\rangle_r,\quad  \langle F(v),F(v)\rangle_r\leq L\langle F(v),v\rangle_r,$$ where $L$ is the Lipschitz constant of $F$.
If moreover
%$$\BB(v,v)= \|v\|^2_{\dot H^{1-s}(\R^d)}, \qquad \forall \, v\in \dot H^{1-s}(\R^d).$$
 $v$ is nonnegative and 
 %$v^{p}\in\dot H^{r}(\mathbb{R}^d)$ for 
 $p\in (1,+\infty)$,   the  following  {Stroock-Varopoulos inequality} holds
\begin{equation}\label{sv}\langle v,v^p\rangle_r \geq \frac{4p}{(p+1)^2}\|v^{(p+1)/2}\|^2_{\dot H^{r}(\R^d)}.\end{equation}

\end{proposition}
{\RRR
\begin{proof}
 The first properties follow at once from the representation \eqref{bilinearform}. 
 %A proof of the Stroock-Varopoulos inequality can be found for instance in \cite{DQRV}.
  \eqref{sv} is also a consequence of \eqref{bilinearform}, by means of the elementary inequality
 \[
 (a-b)(a^p-b^p)\ge\frac{4p}{(p+1)^2}\left(a^{(p+1)/2}-b^{(p+1)/2}\right)^2,
 \]
 holding for any couple of nonnegative numbers $a$, $b$.
\end{proof}
\RRR}

\section{Energy functional and first convergence result}\label{section:functional}

{\RRR From here on it will be always assumed that $d\ge 1$ and $0<s<\min\{1, \frac{d}{2} \}$.}

\subsection{Energy functional}
After noticing that a Borel probability measure $u$ is a tempered distribution with $\hat u$ in $L^1_{loc}(\R^d)$, we
define the energy functional
$\FF_s:\PP_2(\R^d)\to (-\infty,+\infty]$
by
\begin{equation*}\label{functional}
\FF_s(u):=\frac12\|u\|^2_{\dot H^{-s}(\mathbb{R}^d)}
\end{equation*}
We state a basic property of functional $\FF_s$.
\begin{proposition}\label{prop:FP}
The following assertions hold.
\begin{itemize}
\item $D(\FF_s) = \dot H^{-s}(\R^d)\cap\PP_2(\R^d)$.
\item $\FF_s(u)\geq 0$ for every $u\in\PP_2(\R^d)$.
%\item Bounded in $(\PP_2(\R^d), W)$ sublevels of $\FF_s$ are narrowly compact;
\item $\FF_s$ is sequentially lower semicontinuous w.r.t. the narrow convergence.
\end{itemize}
\end{proposition}
\begin{proof}
The first two points are obvious. In order to prove the third one, 
let $\{u_n\}\subset \PP_2(\mathbb{R}^d)$ be a sequence, narrowly converging to $u\in\PP_2(\mathbb{R}^d)$, and such that $\sup_n\FF_s(u_n)<+\infty$.
Using the notation $U_n(\xi):=|\xi|^{-s}\hat u_n(\xi)$, the previous assumption reads as $\sup_n\|U_n\|_{L^2(\R^d)}<+\infty$.
By $L^2$ weak compactness there exists a subsequence of $\{U_n\}$ that weakly converges in $L^2(\R^d)$ to some $U\in L^2(\R^d)$.
By the narrow convergence of $u_n$ we have that $\hat u_n(\xi) \to \hat u(\xi)$ for every $\xi\in\R^d$, and then
 $U_n(\xi)\to  |\xi|^{-s} \hat u (\xi)$ for every $\xi\in\R^d$.
By uniqueness of the weak limits and the lower semicontinuity of the $L^2$ norm we obtain that
 $\FF_s(u)\le\liminf_{n\to\infty}\FF_s(u_n)$ and the  statement holds.
\end{proof}

\subsection{Wasserstein gradient flow, minimizing movements}\label{subsection:MM}

Let  $u_0\in\PP_2(\mathbb{R}^d)$, $\tau>0$.
We let 
\begin{equation}\label{gaussian}
\Gamma_t(x):=\frac{1}{(4\pi t)^{d/2}}e^{-|x|^2/4t},\quad x\in\mathbb{R}^d,\;t>0
\end{equation}
and we define a regularized initial datum as
 \begin{equation}\label{utau0}
 u_\tau^0:=\Gamma_{\omega(\tau)}\ast u_0,\qquad\text{where}\quad \omega(\tau):=\begin{cases} -1/\ln \tau &\mbox{if }\tau\in(0,1/2) \\
-1/\ln (1/2) & \mbox{if }\tau\in[1/2,+\infty).
\end{cases}
 \end{equation}
%
%and  $\omega:(0,+\infty)\to (0,+\infty)$  by
%$$\omega(\tau):=\begin{cases} -1/\ln \tau &\mbox{if }\tau\in(0,1/2) \\
%-1/\ln (1/2) & \mbox{if }\tau\in[1/2,+\infty).
%\end{cases}
%$$
We consider, for $k=1,2,\ldots$, the  problem
\begin{equation}\label{minmov}
\min_{u\in\PP_2(\mathbb{R}^d)} \FF_s(u)+\frac1{2\tau}\,W^2(u,u^{k-1}_\tau).
\end{equation}
\begin{proposition}\label{prop:existenceMM}
For every $\tau>0$ and every $u_0\in \PP_2(\R^d)$ there exists a unique sequence
$\{u_\tau^k:k=0,1,2,\ldots\}\subset D(\FF_s)$
satisfying $u_\tau^0=\Gamma_{\omega(\tau)}\ast u_0$ and such that $u_\tau^k$ is a solution to problem \eqref{minmov} for $k=1,2,\ldots$.
\end{proposition}
\begin{proof}
Let $\tau>0$ and $k\in\N$.
By Proposition \ref{prop:FP} and the properties of the Wasserstein distance,
the functional $u\mapsto  \FF_s(u)+\frac1{2\tau}\,W^2(u,u^{k-1}_\tau)$ is nonnegative,
lower semicontinuous with respect to the narrow convergence and with narrowly compact sublevels.
The existence of minimizers  follows by standard direct methods in calculus of variations.
The uniqueness of minimizers follows from the strict convexity of the functional $ u\mapsto \FF_s(u)+\frac1{2\tau}\,W^2(u,u^{k-1}_\tau)$
with respect to linear convex combinations in $\PP_2(\R^d)$, since $\FF_s$ is a square Hilbert norm.
\end{proof}

By Proposition \ref{prop:existenceMM},  the piecewise constant curve
\begin{equation}\label{PC}
u_\tau(t):=u_\tau^{\lceil t/\tau\rceil},
\end{equation}
is uniquely defined, where $\lceil a\rceil:=\min\{m\in \mathbb{N}:m>a\}$ is the upper integer part.

We say that a curve $u:[0,+\infty)\to \PP_2(\R^d)$ is absolutely continuous with finite energy, and we use the notation
$u\in AC^2([0,+\infty);(\PP_2(\R^d),W))$, if there exists
$m\in L^2([0,+\infty))$ such that $W(u(t_1),u(t_2))\leq \int_{t_1}^{t_2} m(r)\,dr$ for every $t_1,t_2\in[0,+\infty)$, $t_1<t_2$.
If $u\in AC^2([0,+\infty);(\PP_2(\R^d),W))$, then 
there exists its metric derivative defined by
\begin{equation*}\label{defmdw}
		|u'|(t) := \lim_{h \to 0} \frac{W(u(t+h),u(t))}{|h|} \qquad \mbox{for a.e. $t \in [0,+\infty)$},
\end{equation*}
and $|u'|(t)\leq m(t)$ for a.e. $t \in [0,+\infty)$.

\begin{theorem}[First convergence result]\label{th:convergence1}
Let $u_0\in \dot H^{-s}(\R^d)\cap\PP_2(\R^d)$ and $u_\tau$ the piecewise constant curve defined in \eqref{PC}.
For every vanishing sequence $\tau_n$ there exists a subsequence (not relabeled) $\tau_{n}$ and a curve
$u\in AC^2([0,+\infty);(\PP_2(\mathbb{R}^d),W))$ such that
\begin{equation}\label{narrowconv}
u_{\tau_{n}}(t) \to u(t)\quad \mbox{ narrowly as $n\to\infty$, for any $t\in[0,+\infty)$}.
\end{equation}
\end{theorem}
\begin{proof}
The proof is based on  the compactness argument of minimizing movements, stated in  \cite{AGS}. 

Since $0<\hat\Gamma_\tau(\xi)\leq 1$ we have $|\hat u_\tau^0(\xi)|= |\hat\Gamma_{\omega(\tau)}(\xi)\hat u_0(\xi)|\leq |\hat u_0(\xi)|$ and then
\begin{equation}\label{BFs}
 	\FF_s(u_\tau^0)\leq\FF_s(u_0).
\end{equation}
%By Proposition \ref{prop:FP} and the previous inequality we obtain that $\lim_{\tau\to 0}\FF_s(u_\tau^0)=\FF_s(u_0)$.
The first estimate given by the scheme \eqref{minmov}, is the following
\begin{equation}\label{basicestimate}
\FF_s(u_\tau^N) + \frac12\sum_{k=1}^{N} \tau\frac{W^2(u_\tau^k,u_\tau^{k-1})}{\tau^2} \leq \FF_s(u_\tau^0) \leq  \FF_s(u_0), 
\qquad \forall\,N\in\N.
\end{equation}

We show that for any $T>0$ the set $\AA_T:=\{u_\tau^k:\tau>0, N\in\N, N\tau\leq T\}$ is bounded in $(\PP_2(\R^d), W)$ and 
consequently sequentially narrowly compact.\\
Indeed, recalling that $\int_{\R^d}|x|^2\,\d u(x)= W^2(u,\delta_0)$ for any $u\in\PP_2(\R^d)$,
using the triangle inequality and Jensen's discrete inequality we have
\begin{equation}\label{boundW}
\begin{aligned}
   \int_{\R^d}|x|^2u_\tau^N(x)\,\d x= W^2(u_\tau^N,\delta_0) & \leq \Big(\sum_{k=1}^{N} W(u_\tau^k,u_\tau^{k-1}) + W(u_\tau^0,\delta_0)\Big)^2 \\
    &\leq 2\Big(\sum_{k=1}^{N} \tau\frac{W(u_\tau^k,u_\tau^{k-1})}{\tau}\Big)^2 +2 W^2(u_\tau^0,\delta_0) \\
    &\leq 2N\tau \sum_{k=1}^{N} \tau\frac{W^2(u_\tau^k,u_\tau^{k-1})}{\tau^2} +2 W^2(u_\tau^0,\delta_0).
\end{aligned}
\end{equation}
Since for suitable $c>0$ we have
\begin{equation*}
\begin{aligned}
   W^2(u_\tau^0,\delta_0)
    &\leq 2W^2(u_\tau^0,\Gamma_{\omega(\tau)}) +2 W^2(\Gamma_{\omega(\tau)},\delta_0) \\
    &= 2W^2(\Gamma_{\omega(\tau)}\ast u_0,\Gamma_{\omega(\tau)}\ast\delta_0) +2 W^2(\Gamma_{\omega(\tau)},\delta_0) \\
    &\leq  2W^2(u_0,\delta_0) +2 W^2(\Gamma_{\omega(\tau)},\delta_0)
    =  2W^2(u_0,\delta_0) + c\omega(\tau),
\end{aligned}
\end{equation*}
it follows from \eqref{basicestimate} and \eqref{boundW}, since $\FF_s\geq 0$,
that 
\begin{equation}\label{boundA}
	\int_{\R^d}|x|^2u_\tau^N(x)\,\d x \leq 4T\FF_s(u_0) + 4\int_{\R^d}|x|^2u_0(x)\,\d x + 2c,
\end{equation}
and the boundedness of $\AA_T$ follows.

We define the piecewise constant function $m_\tau:[0,+\infty)\to [0,+\infty)$ as
$$ m_\tau(t):=\frac{W(u_\tau(t),u_\tau(t-\tau))}{\tau}$$
with the convention that $u_\tau(t-\tau)=u_\tau(0)$ if $t-\tau<0$.
Since $\FF_s\geq 0$, from \eqref{basicestimate} it follows that
\begin{equation*}\label{basicestimate2}
 \frac12\int_{0}^{+\infty} m_\tau^2(t)\,\d t \leq  \FF_s(u_0).
\end{equation*}
It follows that there exists $m\in L^2(0,+\infty)$ such that $m_\tau$ weakly converges to $m$ in $L^2(0,+\infty)$.
Moreover for any $t_1,t_2 \in [0,+\infty)$,  $t_1<t_2$, setting $k_1(\tau)=[t_1/\tau]$ and $k_2(\tau)=[t_2/\tau]$, 
by triangle inequality it holds
\begin{equation*}\label{equi1}
\begin{aligned}
   W(u_\tau(t_1),u_\tau(t_2)) & \leq \sum_{k=k_1(\tau)}^{k_2(\tau)-1} W(u_\tau^k,u_\tau^{k-1}) \leq \int_{k_1(\tau)\tau}^{k_2(\tau)\tau}m_\tau(t)\,\d t .
   \end{aligned}
\end{equation*}
By the $L^2$ weak convergence of $m_\tau$ the following equicontinuity estimate holds
\begin{equation}\label{equicont}
	\limsup_{\tau\to 0} W(u_\tau(t_1),u_\tau(t_2))\leq \lim_{\tau\to 0} \int_{k_1(\tau)\tau}^{k_2(\tau)\tau}m_\tau(t)\,\d t = \int_{t_1}^{t_2}m(t)\,\d t.
\end{equation}
Applying Proposition 3.3.1 of \cite{AGS} we obtain the convergence  \eqref{narrowconv}.
Passing to the limit in \eqref{equicont} we obtain
\begin{equation*}\label{ac2}
	 W(u(t_1),u(t_2))\leq \int_{t_1}^{t_2}m(t)\,\d t, \qquad \forall\, t_1,t_2 \in [0,+\infty),\quad t_1<t_2,
\end{equation*}
and then $u\in AC^2([0,+\infty);(\PP_2(\mathbb{R}^d),W))$ and
\begin{equation}\label{md}
 	\int_{0}^{+\infty} |u'|^2(t)\,\d t \leq 2 \FF_s(u_0)
\end{equation}
holds.
\end{proof}

\section{Flow interchange and entropy decay estimates}
We briefly review the flow interchange technique introduced by Matthes, McCann  and Savar\'e \cite{MMS}.
Then, with this technique, we  obtain  suitable regularity estimates for solutions to \eqref{minmov}.

\begin{definition}[\textbf{Displacement convex entropy}]\label{entr} 
Let $V:[0,+\infty)\to\R$ be a convex function with super linear growth at infinity, such that $V(0)=0$,  $V\in C^1(0,+\infty)$, 
$V$  is continuous at $0$,
$ \lim_{x\downarrow 0}\frac{V(x)}{x^\alpha}>-\infty$ for some $\alpha>\frac{d}{d+2}$ and
the following McCann displacement convexity assumption (introduced in \cite{Mc}) holds:
\begin{equation*}\label{McCann}
	r \mapsto r^dV(r^{-d}) \qquad \text{is convex and decreasing in }(0,+\infty).
\end{equation*}
If $V$ satisfies the above assumptions, we say that the functional
 $\VV:\PP_2(\R^d)\to (-\infty,+\infty]$, defined by
 $$\VV(u)=\int_{\R^d} V(u(x))\,\d x$$ if $u$
 is absolutely continuous with respect to the Lebesgue measure and $\VV(u)=+\infty$ otherwise,
%  $$\VV(u)=\begin{cases} \int_{\R^d} V(u(x))\,\d x &\mbox{if $u$ is absolutely continuous w.r.t. Lebesgue measure} \\
%+\infty & \mbox{otherwise }
%\end{cases}
%$$ 
is a \emph{displacement convex entropy}.
We say that $V$ is the density function of $\VV$.
% and we assume that $D(\VV)\not=\emptyset$.
\end{definition}
As usual we denote by $D(\VV)$ the set of all $u\in\PP_2(\R^d)$ such that $\VV(u)<+\infty$.

\begin{remark}\label{domain}\rm
The condition on the behavior of $V$ at $0$ is needed as usual to have the integrability of the negative part of {\RRR $V\circ u$},
 as soon as $u$ is a probability density with finite second moment. Moreover, if $u_0\in\PP_2(\mathbb{R}^d)$ and $u_\tau^0$ is the regularization defined by \eqref{utau0}, it is clear that $u_\tau^0\in D(\VV)$ for any displacement convex entropy $\VV$, since $u_\tau^0$ is bounded.   
\end{remark}

It is well known that a displacement convex entropy $\VV$ generates a continuous semigroup $S_t:D(\VV)\to D(\VV)$ satisfying
the following family of \emph{Evolution Variational Inequalities} (see \cite[Theorem 11.2.5]{AGS})
\begin{equation}\label{EVI}
       \frac12 W^2(S_t(u),v)- \frac12 W^2(u,v) \le t(\VV(v)-\VV(S_t(u))) \quad \forall u,v \in D(\VV), \quad  \forall t>0,
\end{equation}
and $S_t(\bar u)$ is the unique distributional solution of the Cauchy problem $$\de_t u=\Delta(L_V(u)),\quad u(0)=\bar u,$$ where
$L_V(u):=uV'(u)-V(u)$, such that \eqref{EVI} holds. 
The semigroup is contractive w.r.t. $W$ and extends to $\overline{D(\VV)}=\PP_2(\mathbb{R}^d)$.
Thanks to the regularizing effect $S_t(u)\in D(\VV)$ for any $u\in\PP_2(\mathbb{R}^d)$ and any $t>0$, 
we obtain that \eqref{EVI} holds for every $u,v\in\PP_2(\mathbb{R}^d)$.

If $u\in D(\FF_s)$ we define the dissipation of $\FF_s$ along the flow $S_t$ of $\VV$ by
\begin{equation*}
 \mathfrak{D}_\VV \FF_s(u):= \limsup_{t \downarrow 0} \frac{\FF_s(u)-\FF_s(S_t(u))}{t}.
\end{equation*}

\begin{proposition}[\textbf{Flow interchange}]\label{prop:FI}
Let $\{u_\tau^k:k=0,1,2,\ldots\}$ be the sequence given by {\rm Proposition \ref{prop:existenceMM}} and $\VV$  a displacement convex entropy.
If 
\begin{equation}\label{>-inf}
\mathfrak{D}_\VV\FF_s(u_\tau^k)>-\infty\quad \text{for } k\geq 1,
\end{equation} 
then $u_\tau^k\in D(\VV)$ and
\begin{equation*}\label{mms09}
 \mathfrak{D}_\VV \FF_s(u_{\tau}^k)\le   \frac{\VV(u_{\tau}^{k-1}) -\VV(u_{\tau}^k)}{\tau}, \qquad \, k= 1,2,\ldots.
\end{equation*}
\end{proposition}

\begin{proof}
We have $u_\tau^0\in D(\VV)$, see Remark \ref{domain}. 
For  $t>0$ and $k>0$,
by definition of minimizer there holds
\[
    \FF_s(u_\tau^k) + \frac{1}{2\tau}W^2(u_\tau^k,u_\tau^{k-1})\leq \FF_s(S_t(u_\tau^k)) + \frac{1}{2\tau}W^2(S_t(u_\tau^k),u_\tau^{k-1}),
\]
that is,
\[
   \tau(\FF_s(u_\tau^k) -\FF_s(S_t(u_\tau^k)) ) \leq \frac{1}{2}W^2(S_t(u_\tau^k),u_\tau^{k-1}) - \frac{1}{2}W^2(u_\tau^k,u_\tau^{k-1}).
\]
By using \eqref{EVI} we obtain
\[
   \tau\frac{\FF_s(u_\tau^k) -\FF_s(S_t(u_\tau^k))}{t} \leq \VV(u_\tau^{k-1}) - \VV(S_t(u_\tau^k)).
\]
As $u_\tau^0\in D(\VV)$, we may now recursively apply the above inequality: thanks to \eqref{>-inf}, by passing to the limit as $t\downarrow 0$ and using the lower semicontinuity of $\VV$ with respect to the narrow convergence we conclude.
\end{proof}

\begin{remark}\rm With the next lemmas we will characterize the dissipation and show that \eqref{>-inf} holds true for any displacement convex entropy $\VV$. 
\end{remark}

\subsection{Improved regularity}
The following result makes use of flow interchange
with the choice $\VV=\HH$, the entropy functional.

%\begin{lemma}\label{lemma:diff}
%Let $u\in D(\FF_s)$ and $S_t$ the heat semigroup.
%If $g:[0,+\infty)\to \R$ is defined by $g(t)=\FF_s(S_t(u))$,
%then $g \in C^0([0,+\infty))$, $g$ is differentiable in $(0,+\infty)$ and
%\begin{equation}\label{gprime}
%g'(t) = -\frac{1}{(2\pi)^d} \int_{\R^d}|\xi|^{2(1-s)}|\widehat{S_t(u)}(\xi)|^2\d \xi \qquad \forall t\in (0,+\infty).
%\end{equation}
%\end{lemma}
%\begin{proof}
%We use the notation $u_t:=S_t(u)$.
%By uniqueness of the solution of the heat equation the representation
%$u_t=\Gamma_t\ast u$ holds, where $\Gamma_t$ denotes the family of gaussian kernels.
%At the level of Fourier transform the equation reads
%$\partial_t \hat u_t(\xi) + |\xi|^2\hat u_t(\xi)=0$ in $\R^d\times(0,+\infty)$.
%
%Taking into account the smoothness of $u_t$ with respect to time and space we obtain
%\begin{equation}
%\begin{aligned}
%    \frac{\d}{\d t}\FF_s(u_t) &= \frac{1}{(2\pi)^d} \int_{\R^d}|\xi|^{-2s}\hat u_t(\xi)\de_t\bar{\hat u}_t(\xi) \,\d\xi \\ &=
%    -\frac{1}{(2\pi)^d} \int_{\R^d}|\xi|^{-2s}\hat u_t(\xi)|\xi|^2\bar{\hat u}_t(\xi) \,\d\xi
%    \end{aligned}
%\end{equation}
%and the differentiability of $g$ and \eqref{gprime} follows.
%
%Since $\hat\Gamma_t(\xi)\leq 1$ we have $|\hat u_t(\xi)|^2= |\hat\Gamma_t(\xi)\hat u(\xi)|^2\leq |\hat u(\xi)|^2$ and
%it follows that $\FF_s(u_t)\leq \FF_s(u)$.
%Since $\FF_s$ is lower semi continuous with respect to the narrow convergence, the continuity of $g$ at $0$ follows.
%\end{proof}

\begin{lemma} \label{lemma:ed}
Let $u_0\in D(\FF_s)$ and $\{u_\tau^k:k=0,1,2,\ldots\}$ the sequence given by  {\rm Proposition \ref{prop:existenceMM}}.
Then
%\begin{equation}\label{decayentropy}
%    \HH(u_\tau^k) \leq \HH(u_\tau^{k-1}), \qquad k=1,2,3,...
%\end{equation}
%Moreover
$u_\tau^k\in \dot H^{1-s}(\R^d)\cap D(\mathcal{H})$ for any $k\ge 0$ and
\begin{equation}\label{regestimate}
   \|u_\tau^k\|^2_{\dot H^{1-s}(\R^d)} \leq  \frac{\HH(u_\tau^{k-1})-\HH(u_\tau^k)}{\tau},  \qquad k=1,2,\ldots.
   \end{equation}
In particular,
\[
\mathcal{H}(u_\tau^k)\le \mathcal H(u_\tau^{k-1}), \qquad k=1,2,\ldots.
\]
\end{lemma}

\begin{proof}
By its definition in \eqref{utau0}, it is clear that $u_\tau^0\in D(\mathcal{H})$.

We denote by $S_t$ the heat semigroup on $\R^d$, namely the flow generated by the entropy $\mathcal H$.
For $k\geq 0$ we
have $S_t(u_\tau^k)\in \dot H^{1-s}(\mathbb{R}^d)$ for any $t>0$. Indeed, by uniqueness of the solution of the heat equation the representation
$S_t(u_\tau^k)=\Gamma_t\ast u_\tau^k$ holds, where $\Gamma_t$ denotes the family of gaussian kernels \eqref{gaussian}. 
Then, using the notation $w_t:=S_t(u_\tau^k)$, since $\hat \Gamma_t$ is a Gaussian, by \eqref{basicestimate} we have
\[\begin{aligned}
\int_{\mathbb{R}^d}|\xi|^{2(1-s)}|\hat w_t(\xi)|^2\,\d\xi&=\int_{\mathbb{R}^d}|\xi|^{2(1-s)}|\hat\Gamma_t(\xi)|^2|
\hat {\RRR u}_\tau^k(\xi)|^2\,\d\xi\\&\le C_t\|u_\tau^k\|^2_{\dot H^{-s}(\mathbb{R}^d)}\le 2C_t\FF_s(u_\tau^k)\le 2C_t \FF_s(u_0)<+\infty,
\end{aligned}\]
where $C_t:=\max_{\xi\in\mathbb{R}^d}|\xi|^2|\hat\Gamma_t(\xi)|^2$. 
Since $u_{\tau}^0 := \Gamma_{\omega(\tau)}\ast u_0$ (see \eqref{utau0}), 
a similar argument shows that $u_\tau^0\in \dot H^{1-s}(\mathbb{R}^d)$.

Next we let $k>0$ and  we consider the real function 
$t \mapsto \FF_s(w_t)$ for $t\in [0,+\infty)$. We claim that this function is differentiable in $(0,+\infty)$ and
continuous at $t=0$, and that 
% define $g:[0,+\infty)\to \R$ by $g(t):=\FF_s(S_t(u_\tau^k))$. We claim that
%$g$ is differentiable in $(0,+\infty)$ and
\begin{equation}\label{gprime}
%g'(t)
\frac{d}{dt} \FF_s(w_t) = -\|S_{t}(u_\tau^k)\|^2_{\dot H^{1-s}(\R^d)} = -\| w_t\|^2_{\dot H^{1-s}(\R^d)}
 \qquad \forall\,t\in(0,+\infty).
%-\frac{1}{(2\pi)^d} \int_{\R^d}|\xi|^{2(1-s)}|\widehat{S_t(u)}(\xi)|^2\d \xi \qquad \forall t\in (0,+\infty).
\end{equation}
To show this we recall that in Fourier variables
%Indeed, 
%in terms of Fourier transform 
the heat equation reads
$\partial_t \hat w_t(\xi) + |\xi|^2\hat w_t(\xi)=0$ in $\R^d\times(0,+\infty)$.
Taking into account the smoothness of $w_t$ we obtain %with respect to time and space we obtain
\begin{equation*}
\begin{aligned}
    \frac{\d}{\d t}\FF_s(w_t) &=  \frac{1}{2(2\pi)^d} \frac{\d}{\d t}  \int_{\R^d}|\xi|^{-2s}\overline{\hat w_t(\xi)}{{\hat w_t(\xi)}} \,\d\xi=\frac{1}{(2\pi)^d} \int_{\R^d}|\xi|^{-2s}\overline{\hat w_t(\xi)}\de_t{{\hat w_t(\xi)}} \,\d\xi \\ &=
    -\frac{1}{(2\pi)^d} \int_{\R^d}|\xi|^{-2s}\overline{\hat w_t(\xi)}|\xi|^2{\hat w}_t(\xi) \,\d\xi=-\|w_t\|^2_{\dot H^{1-s}(\mathbb{R}^d)}
    \end{aligned}
\end{equation*}
and thus the desired differentiability and \eqref{gprime} follow. Now, we prove that 
the map $t\mapsto \FF_s(w_t)$ is continuous at $t = 0$.
Indeed, since $0< \hat\Gamma_t(\xi)\leq 1$ we have $|\hat w_t(\xi)|^2= |\hat\Gamma_t(\xi)\hat {\RRR u}_\tau^k(\xi)|^2\leq |\hat {\RRR u}_\tau^k(\xi)|^2$ and
it follows that $\FF_s(w_t)\leq \FF_s({\RRR u}_\tau^k)$.
Since $\FF_s$ is lower semi continuous with respect to the narrow convergence, the continuity at $0$ follows.

By Lagrange's mean value Theorem, for every $t>0$ there exists $\theta(t)\in (0,t)$ such that
\[
\frac{\FF_s(u_\tau^k) -\FF_s(S_t(u_\tau^k))}{t} =  \|S_{\theta(t)}(u_\tau^k)\|^2_{\dot H^{1-s}(\R^d)}.
\]
By the lower semicontinuity of the $\dot H^{1-s}$ norm with respect to the narrow convergence it follows that
\[
 \|u_\tau^k\|^2_{\dot H^{1-s}(\R^d)} \leq   \mathfrak{D}_\HH \FF_s(u_\tau^k).
\]
 Then, by Proposition \ref{prop:FI},   we obtain that $u_\tau^k\in D(\mathcal{H})\cap \dot H^{1-s}(\mathbb{R}^d)$ and \eqref{regestimate} holds.
%Then it follows
%\[
%    \HH(S_t(u_\tau^k))) \leq \HH(u_\tau^{k-1}),
%\]
%and \eqref{decayentropy} is a consequence of the lower semicontinuity  of $\HH$ with respect to the narrow convergence.
%Since $\hat u(\xi)\leq 1$ for any $\xi$ and
%$(\alpha+|\xi|^2)^{1-s}\leq (\alpha+1)|\xi|^2(\alpha+|\xi|^2)^{-s}$ for every $\xi$ such that $|\xi|\geq 1$,
%we have
%\[
%     \int_{\R^d}(\alpha+|\xi|^2)^{1-s}|\widehat{S_t(u)}(\xi)|^2\d \xi \leq (\alpha+1)|B_1|+ (\alpha+1) \int_{\R^d}(\alpha+|\xi|^2)^{-s}|\xi|^2|\widehat{S_t(u)}(\xi)|^2\d \xi.
%\]
%It follows that there exists a constant $C$ depending only on $\alpha$ and $d$ such that
%\[
%    \tau \|S_{\theta(t)}(u_\tau^k)\|^2_{H^{1-s}(\R^d)} \leq \tau C + C \HH(u_\tau^{k-1}) - \HH(S_t(u_\tau^k)).
%\]
%Since
%\[
%    \tau \|S_{\theta(t)}(u_\tau^k)\|^2_{\dot H^{1-s}(\R^d)} \leq \HH(u_\tau^{k-1}) - \HH(S_t(u_\tau^k)),
%\]
%passing to the limit as $t \to 0$ and taking into account the lower semicontinuity of $\HH$ and the $\dot H^{1-s}$ norm with respect to the narrow convergence
%we obtain that $u_\tau^k\in \dot H^{1-s}(\R^d)$ and \eqref{regestimate} holds.
\end{proof}

%In the following, $u_{\tau}(t):=u_{\tau}^{\lceil t/\tau\rceil}$ denotes the discrete piecewise constant curve.

Integrating the estimate \eqref{regestimate}  with respect to time, we obtain the following space-time bound on the discrete solution $u_\tau$.
For the integer part of the real number $a$ we use the notation $[a]:=\max\{m\in\mathbb{Z}:m\le a\}$. 

\begin{corollary}\label{furthercorollary}
Let $u_0\in  D(\FF_s)$, $\{u_\tau^k:k=0,1,2,\ldots\}$ the sequence given by {\rm Proposition \ref{prop:existenceMM}} and
$u_{\tau}$ the corresponding discrete piecewise constant approximate solution defined in \eqref{PC}.
Then
$u_{\tau}(t)\in \dot H^{1-s}(\mathbb{R}^d)$ for every $t>0$ and
\begin{equation}\label{RI1}
\int_{T_0}^T\|u_{\tau}(t)\|_{ \dot H^{1-s}(\mathbb{R}^d)}^2\,\d t \le \HH(u_\tau^{N_0(\tau)}) + c\Big(1 +T\FF_s(u_0)+\int_{\R^d}|x|^2\,\d u_0(x) \Big)
\end{equation} 
holds for any $T_0\geq0$ and $T>T_0$,
where $N_0(\tau):=[ T_0/\tau]$ and $c$ is a constant depending only on the dimension $d$.
\end{corollary}
\begin{proof}
Let $T>0$, $N= \lceil T/\tau\rceil$ and $N_0=N_0(\tau)$. By \eqref{regestimate} we obtain
$$\int_{T_0}^T\|u_{\tau}(t)\|_{\dot H^{1-s}(\mathbb{R}^d)}^2\,\d t \leq \sum_{k={N_0+1}}^{N}\tau\|u^k_{\tau}\|_{\dot H^{1-s}(\mathbb{R}^d)}^2
\leq  \HH(u_\tau^{N_0}) -\HH(u_\tau^N).$$
By a Carleman type inequality there holds
$$ -\HH(u_\tau^N)\leq \tilde c\Big(1+\int_{\R^d}|x|^2u_\tau^N(x)\,\d x\Big)$$
for a suitable constant depending only on $d$.
From \eqref{boundA}  we obtain
$$ -\HH(u_\tau^N)\leq c\Big(1+ T\FF_s(u_0)+\int_{\R^d}|x|^2 \,\d u_0(x) \Big)$$
for $c$ depending only on the dimension $d$ and we conclude.
\end{proof}

%By Sobolev embedding  \eqref{fractionalembedding} we obtain that $u_{\tau,\eps}\in L^{2^@}(\mathbb{R}^d)\cap\PP_2(\mathbb{R}^d)$ where $$2^@:=\frac{2d}{d-2+2s}>2.$$ In particular $u_{\tau,\eps}\in H^{1-s}(\mathbb{R}^d)$.

\subsection{Decay of the entropies}

In the next Lemma we apply the flow interchange to a general displacement convex entropy $\GG$ and we 
compute a lower bound for the dissipation of the functional $\FF_s$ along the flow of $\GG$.
This result is useful for the regularizing effect and the $L^p$ estimates.
\begin{lemma}\label{lemma:decay1}
Let $u_0\in D(\FF_s)$ and $\{u_\tau^k:k=0,1,2,\ldots\}$ the sequence given by  {\rm Proposition \ref{prop:existenceMM}}.
Let $\GG$ be  a displacement convex entropy with density function $G$, according to  {\rm Definition \ref{entr}}. Then $u_\tau^k\in D(\GG)$ for any $k\ge 0$ and  there holds
%Let $G:[0,+\infty)\to \R$ be a  a convex function with super linear growth, $G\in C^1(0,+\infty)$ and continuous at $0$,
%satisfying the McCann displacement convexity assumption \eqref{McCann}.
%Denoting by $\GG(u):=\int_{\R^d}G(u(x))\,\d x$ the associated integral functional we have that
\begin{equation}\label{decayLp}
0\le \langle u_\tau^k,L_G(u_\tau^k)\rangle_{1-s}\le\frac{\GG(u_\tau^{k-1})-\GG(u_\tau^k)}{\tau},\qquad k=1,2,\ldots.
\end{equation}
In particular,
\begin{equation*}
    \GG(u_\tau^k) \leq \GG(u_\tau^{k-1}),\qquad k=1,2,\ldots.
\end{equation*}
\end{lemma}

\begin{proof}
The proof is based on the same argument of Lemma \ref{lemma:ed}. First of all, we have $u_\tau^0\in D(\GG)$ by Remark \ref{domain}.

%First of all we observe that $\GG(u_\tau^0)<+\infty$.
For $\eps>0$ we consider the displacement convex entropy
$$\VV(u) := \GG(u) + \eps\HH(u).$$
We denote by $S_t$ the flow associated to $\VV$ with respect to the Wasserstein distance.
Let us fix $k>0$ and define $w_t:=S_t(u_\tau^k)$, thus 
$w_t$ satisfies the equation
\begin{equation}\label{nonlinear}
 \de_t w_t=\Delta L_G(w_t)+\eps \Delta w_t  = \Delta \Psi(w_t), 
 \end{equation}
with initial datum $u_\tau^k$, where $L_G(v)=vG'(v)-G(v)$ and $\Psi(v) = L_G (v) + \eps v$.
%Thanks to  the assumptions on $G$, $\VV$ is displacement convex and Proposition \ref{prop:FI} can be applied.
%{\color{blue} 
Equation \eqref{nonlinear} is a quasilinear non degenerate parabolic equation since
$\Psi$ satisfies $\Psi'>0$. As a result, the solution $w_t$ is bounded,  smooth  and strictly positive for $t>0$
(see for example \cite[Chapter 3]{V}). 
Moreover since Lemma \ref{lemma:ed} gives $u_\tau^k\in \dot H^{1-s}(\mathbb{R}^d)$ for any $k> 0$ 
and $u_\tau^k\in L^1(\R^d)$ by construction, we have that $u_\tau^k\in L^2(\R^d)$ thanks to 
the Sobolev embedding \eqref{fractionalembedding}. 
Now, if we test equation \eqref{nonlinear} with $w_t$, we immediately get (recall that $L_G$ is monotone increasing) 
\begin{equation}\label{L2af}
\|w_t \|_{L^2(\mathbb{R}^d)}\le \| u_\tau^k\|_{L^2(\mathbb{R}^d)}, \,\,\,\,\forall t>0.
\end{equation}
Thus, the estimate above combined with the lower semi continuity of the norm, gives the
strong continuity in $L^2(\R^d)$ of the semigroup. 
%Moreover, testing the equation \eqref{nonlinear} with $\Lsinv w_t$ and using the Strook Varopoulos inequality
%in Proposition \ref{prop:BB}, we get
%\begin{eqnarray}
%\frac{1}{2}\frac{d}{dt}\int_{\R^d} \vert(-\Delta)^{-\frac{s}{2}}w_t\vert^2 \d x + \eps\int_{\R^d}\vert (-\Delta)^{\frac{1-s}{2}}w_t\vert^2 \d x\le 0,
%\end{eqnarray}
%which readily implies that the map $t\mapsto \FF_s(w_t) = \| (-\Delta)^{\frac{s}{2}}w_t\|$ is continuous at $t=0$ and 
%that $\int_{\R^d}\vert (-\Delta)^{\frac{1-s}{2}}w_t\vert^2 \d x$ is finite for almost any $t\in (0,+\infty)$. 

%as shown in Lemma \ref{lemma:ed}, and then $u_t\in \dot H^{1-s}(\mathbb{R}^d)$ for $t>0$.
%Since $u_\tau^k\in L^1(\R^d)$ as well, the Sobolev embedding \eqref{fractionalembedding} 
%implies  $u_\tau^k\in L^2(\R^d)$. Therefore, $S_t(u_\tau^k)\to u_\tau^k$ strongly in $L^2(\mathbb{R}^d)$ as $t\to 0$ (the semigroup $S_t$ is strongly continuous in $L^2(\mathbb{R}^d)$).
%Now, 
%or $t\in[0,+\infty)$ and $k>0$
%we define $\Phi(t):=\FF_s(S_t(u_\tau^k))$, with $k>0$.
By making use of the transformed version of \eqref{nonlinear}, there holds, for any $t>0$,
$$
\begin{aligned}
\frac{d}{dt}\FF_s (w_t) &=\frac{1}{(2\pi)^d} \int_{\R^d}|\xi|^{-2s}\overline{\hat w_t(\xi)}\de_t{{\hat w_t(\xi)}} \,\d\xi\\&=- \frac{1}{(2\pi)^d}\int_{\mathbb{R}^d}|\xi|^{-2s}\overline{\hat w_t(\xi)}|\xi|^2\left(\widehat{L_G(w_t)}(\xi)+\eps\hat w_t(\xi)\right)\\
	&= -  \langle w_t,L_G(w_t)\rangle_{1-s} -\eps  \langle w_t,w_t\rangle_{1-s}.
\end{aligned}	
	$$
Notice that $L_G$ is non decreasing and locally Lipschitz, and since $w_t$ is bounded and  $w_t \in H^{1-s}(\R^d)$ for
 $t\in (0,+\infty)$,
from Proposition \ref{prop:BB} we obtain $L_G\circ w_t\in \dot H^{1-s}(\mathbb{R}^d)$ and 
$\langle w_t,L_G(w_t)\rangle_{1-s}\ge 0 $ for $t$ in $(0,+\infty)$. 
In particular, $t\mapsto \FF_s(w_t)$ is differentiable in $(0,+\infty)$.
% Moreover, 
%f we integrate the equation above with respect to time, we get
%\[
%\FF_s(w_t) \le \FF_s(u^{k}_{\tau}), \,\,\hbox{ for }t>0.
%\] 
%As a consequence, we get that $\limsup_{t\searrow 0}\FF_s(w_t) \le \FF_s(u^{k}_{\tau})$ which gives
%the continuity at $t=0$ of $\FF_s(w_t)$.
%Note that $u_\tau^0$ is bounded as well, then still by Proposition \ref{prop:BB} we have $L_G\circ u_\tau^0\in\dot H^{1-s}(\mathbb{R}^d)$.
%  $u_\tau^0\in L^2(\R^d)$ because it belongs to $L^\infty\cap L^1(\R^d)$,

Next we shall prove that $t\mapsto \FF_s(w_t)$ is continuous at $t = 0$.
Since $w_t$ is a probability density, we have that $\vert \hat{w}_t(\xi)\vert\le 1$ for any $\xi\in \R^d$.
Thus,  
%Taking into account that $|\hat u_t(\xi)|\leq 1$ 
for every $t\in[0,+\infty)$ and for some $\delta>0$ we have
\begin{equation*}\begin{aligned} 
\|S_t(u_\tau^k)\|^2_{\dot H^{-s-\delta}(\R^d)} &= \int_{\R^d}|\xi|^{-2s-2\delta}|\hat {\RRR w}_t(\xi)|^2\,\d\xi \\
&\leq \int_{\{|\xi|\geq1\}}|\hat {\RRR w}_t(\xi)|^2\,\d\xi + \int_{\{|\xi|<1\}}|\xi|^{-2s-2\delta}\,\d\xi.
\end{aligned}\end{equation*}
By \eqref{L2af} and Plancherel's Theorem,
for $0<\delta<d/2-s$ the previous estimate shows that 
$\|S_t(u_\tau^k)\|_{\dot H^{-s-\delta}(\R^d)}\leq c$ for every $t\in[0,1]$, where $c$ is a constant not depending on $t$.
Then, for a suitable $\theta\in(0,1)$ and $0<\delta<d/2-s$, by interpolation we have
$$
\begin{aligned}
 \|S_t(u_\tau^k) -u_\tau^k\|_{\dot H^{-s}(\R^d)} &\leq  
 \|S_t(u_\tau^k) -u_\tau^k\|_{L^{2}(\R^d)}^{1-\theta}  \|S_t(u_\tau^k) -u_\tau^k\|_{\dot H^{-s-\delta}(\R^d)}^\theta \\&\leq (2c)^\theta  \|S_t(u_\tau^k) -u_\tau^k\|_{L^{2}(\R^d)}^{1-\theta}, 
\end{aligned}
 $$
and the obtained $L^2(\R^d)$ strong continuity of $S_t$ implies that $t\mapsto \FF_s(w_t)$ is continuous at $t= 0$.

By the same argument of Lemma \ref{lemma:ed}, based on Lagrange mean value theorem, we  obtain for suitable $\theta(t)\in (0,t)$
\[
\frac{\FF_s(u_\tau^k) -\FF_s(S_t(u_\tau^k))}{t} =  	\eps\|S_{\theta(t)}(u_\tau^k)\|^2_{\dot H^{1-s}(\R^d)}+\langle S_{\theta(t)}(u_\tau^k),L_G(S_{\theta(t)}(u_\tau^k))\rangle_{1-s}.
\]
Notice that the map $u\mapsto \langle u,L_G(u)\rangle_{1-s}$ is lower semicontinuous with respect to the strong $L^2(\mathbb{R}^d)$ convergence. This follows by applying Fatou's lemma to the expression \eqref{bilinearform}, where the integrand is nonnegative in this case, since $L_G$ is nondecreasing (see Proposition \ref{prop:BB}). Therefore, by passing to the limit as $t\downarrow 0$ we obtain
\[
0\le  	%\eps\|u_\tau^k\|^2_{\dot H^{1-s}(\R^d)}+
	\langle u_\tau^k,L_G(u_\tau^k)\rangle_{1-s}\le \liminf_{t\downarrow 0} \frac{\FF_s(u_\tau^k) -\FF_s(S_t(u_\tau^k))}{t} \le\mathfrak{D}_{\mathcal{V}}\FF_s(u_\tau^k).
\]
The latter estimate, together with Proposition \ref{prop:FI}, entails $u_\tau^k\in D(\VV)$ and
$$ %\eps\tau\|u_\tau^k\|^2_{\dot H^{1-s}(\mathbb{R}^d)}+
 \tau\langle u_\tau^k,L_G(u_\tau^k)\rangle_{1-s}+\GG(u_\tau^k) +\eps\HH(u_\tau^k) \leq \GG(u_\tau^{k-1}) +\eps\HH(u_\tau^{k-1}), \qquad k=1,2,\ldots. $$
 In particular,  for $k=1,2,\ldots$ there is $u_\tau^k\in D(\GG) $ and $\langle u_\tau^k,L_G(u_\tau^k)\rangle_{1-s}<+\infty$.
 %where the first term is nonnegative. Lemma \ref{lemma:ed} shows that $u_\tau^k\in D(\mathcal{H})$ for any $k\ge 0$.
 By letting $\eps\to 0$ we find that  \eqref{decayLp} holds.
% }
\end{proof}

\subsection{Regularizing effect}

In order to obtain a quantitative decay of a positive logarithmic entropy and of the $L^p$ norms of the discrete solution
we need the two following propositions.

\begin{proposition}\label{prop:DI}
Let $\phi:\R\to\R$ be a convex $C^1$ function and $\tau>0$.
If $a_k$ and $b_k$ satisfies
$$ a_k-a_{k-1} \leq -\tau\phi'(a_k), \qquad b_k-b_{k-1} = -\tau\phi'(b_k), \quad \forall k\in \N$$
and $a_0\leq b_0$, then $a_k\leq b_k$ for every $k\in\N$.
\end{proposition}

\begin{proof}
By induction, assuming that $a_{k-1}\leq b_{k-1}$ we have that
$$ a_k+\tau\phi'(a_k) \leq a_{k-1}\leq  b_{k-1} = b_k+\tau\phi'(b_k).$$
Since the function $r\mapsto r+\tau\phi'(r)$ is strictly increasing we conclude.
\end{proof}

\begin{proposition}\label{prop:CL}
Let $\phi:\R\to\R$ be a convex $C^1$ function and $\tau>0$.
Let $b_0\in \R$ and $b_k$ be satisfying
$$ b_k-b_{k-1} = -\tau\phi'(b_k), \quad \forall k\in \N$$
and $b:[0,+\infty)\to\R$ the solution of the Cauchy problem
\begin{equation}\label{CauchyPb}
 b'(t)=-\phi'(b(t)), \qquad b(0)=b_0.
\end{equation}
Then $|b_k-b(k\tau)|\leq \frac{1}{\sqrt{2}}|\phi'(b_0)|\tau$.
\end{proposition}

\begin{proof}
The result is the error estimate for the Euler implicit discretization scheme. 
See for instance the general expression derived by Nochetto-Savar\'e-Verdi \cite{NSV} and 
\cite[Theorem 4.0.7]{AGS}.
\end{proof}

In the following of the paper we denote by $\KK:\PP_2(\R^d)\to [0,+\infty]$
the positive entropy defined by $\KK(u):=\int_{\R^d}u(x)\log (u(x)+1) \,\d x$ if $u$ is absolutely continuous 
with respect to the Lebesgue measure and
$\KK(u)=+\infty$ otherwise, which is a displacement convex entropy according to Definition \ref{entr}.

\begin{lemma}\label{lemma:decay2} Let $\{u_\tau^k:k=0,1,2,\ldots\}$ be the sequence given by  {\rm Proposition \ref{prop:existenceMM}}.
There holds %exists a constant $C_0$ depending only on $s$ and $d$ such that
\begin{equation}\label{discreteREK}
    \KK(u_\tau^k) \leq \min \{ \KK(u_\tau^0),C_0(k\tau)^{-\gamma_0}\} + \tfrac{\tilde C_0}{\sqrt2}\,\tau (\KK(u_\tau^0))^{\beta_0}     ,\qquad k=1,2,\ldots,
\end{equation}
where $\gamma_0:=\frac{1}{2}\frac{d}{d+2(1-s)}$,  $\beta_0:=\frac{3d+4(1-s)}{d}$, $\tilde C_0:=2^{-\frac{3d+4(1-s)}{d}}{\RRR A_{d,s}}$,
%$\tilde C_0:=2^{-\frac{3d+4(1-s)}{d}}S_{d,1-s}^{-2}$,
$C_0=(\tilde C_0(\beta_0-1))^{-\gamma_0}$, {\RRR  $A_{d,s}:=S_{d,1-s}^{-2}$ if $d\ge 2$, $A_{1,s}:=S_{1,\frac{1-s}{4-2s}}^{2s-2}$ and $S_{d,r}$ is defined by \eqref{best}.}

\noindent Moreover, for every $p\in(1,+\infty)$ there holds % exists a constant $C_p$ depending only on $s$, $d$ and $p$ such that
\begin{equation}\label{discreteRELp}
    \|u_\tau^k\|^p_{L^p(\R^d)} \leq \min \{ \|u_\tau^0\|^p_{L^p(\R^d)},C_p^p\,(k\tau)^{-p\gamma_p}\}  
    + \tfrac{\tilde C_p}{\sqrt{2}}\,\tau\,\|u_\tau^0\|^{p\beta_p}_{L^p(\R^d)}   ,\quad k=1,2,\ldots,
\end{equation}
where $\gamma_p:=\frac{p-1}{p}\frac{d}{d+2(1-s)}$,   $\beta_p:=\frac{pd+2(1-s)}{(p-1)d}$, $\tilde C_p:=\frac{4p(p-1)}{(p+1)^2}{\RRR B_{d,s}}$,
%$\tilde C_p:=\frac{4p(p-1)}{(p+1)^2S^2_{d,1-s}}$,
 $C_p:=(\tilde C_p(\beta_p-1))^{-\gamma_p}$, {\RRR  $B_{d,s}:=S_{d,1-s}^{-2}$ if $d\ge 2$ and $B_{1,s}:=S_{1,\frac{1-s}{4-2s}}^{4s-8}$.}

\end{lemma}

\begin{proof}

%{\RRR Let us prove the result for  $d\ge2$. The modifications to be done in case $d=1$ will be given in Remark \ref{rem:oned} below.}
%\noindent
%{\RRR {\bf The case $d\ge 2$.}}
We shall apply Lemma \ref{lemma:decay1} to the particular cases $\GG=\KK$ and $\GG=\GG_p$, where $\GG_p$ is the displacement convex entropy with power density function $G_p(u) = \frac{1}{p-1}u^p$, for $p\in(1,+\infty)$. 

Let us start with $\GG=\KK$, so that the density function is $G(u)=u(\log u+1)$.
%Let $G(u)=u\ln(u+1)$,  $\KK(u)=\int_{\R^d} G(u(x))\, \d x$ and
%$\VV=\KK+\eps\HH$.
%%The gradient flow of $\VV$ with respect to the Wasserstein distance
%%satisfies the EVI and the equation
%%$ \de_t u=\Delta L_G(u)+\eps \Delta u$.
%Let $S_t$ be the flow associated to $\VV$ with respect to the Wasserstein distance and
%$u_t:=S_t(u_\tau^k)$, $g(t)=\FF_s(u_t)$.
%
%We have to compute and estimate the dissipation of $\FF_s$ along the flow of $\VV$.
In this case $L_G(u):=uG'(u)-G(u)= \frac{u^2}{u+1}$.
Since $L_G$ is increasing on $[0,+\infty)$ and $L_G'(u)=\frac{u}{u+1}<1$ by Proposition \ref{prop:BB}
we have, for any $k\in\mathbb{N}$,
\begin{equation}\label{scalarestimate} 
	+\infty>\langle u_\tau^k,L_G(u_\tau^k)\rangle_{1-s} \geq  \langle L_G(u_\tau^k),L_G(u_\tau^k)\rangle_{1-s} 
	= \|L_G(u_\tau^k)\|^2_{\dot H^{1-s}(\R^d)} .
\end{equation}
Since $0\leq L_G(u) < u$ we have $\|L_G(u_\tau^k)\|_{L^1(\R^d)}\leq \|u_\tau^k\|_{L^1(\R^d)} =1$. 
Therefore $L_G\circ u_\tau^k\in \dot H^{1-s}(\mathbb{R}^d)\cap L^1(\mathbb{R}^d)$.
{\RRR Using \eqref{FractionalGNS} in the case $d\ge 2$ and \eqref{eq:Bs} in the case $d=1$ we obtain
\begin{equation}\label{scalar2} \|L_G(u_\tau^k)\|^2_{\dot H^{1-s}(\R^d)} \geq A_{d,s} \int_{\R^d} (L_G(u_\tau^k))^q\,\d x\end{equation}	
for $q:=2+2(1-s)/d$, where $A_{d,s}:=S_{d,1-s}^{-2}$ if $d\ge 2$, $A_{1,s}:=S_{1,\frac{1-s}{4-2s}}^{2s-2}$.}
By Jensen inequality we have
$$  \int_{\R^d} (L_G(u_\tau^k))^q\,\d x = \int_{\R^d} \frac{(u_\tau^k)^{2q-1}}{(u_\tau^k+1)^q}\,u_\tau^k\,\d x\geq\Big(\int_{\R^d} \frac{u_\tau^k}{(u_\tau^k+1)^{q/(2q-1)}}\,u_\tau^k\,\d x \Big)^{2q-1}
%\geq C_0(\KK(u))^{2q-1},
$$
and an elementary computation shows that, for any $u\in [0,+\infty)$, there holds
$$\frac{2u^2}{(u+1)^{q/(2q-1)}} \geq  u\ln(u+1),$$
then we have
\begin{equation}\label{scalar3}
\int_{\mathbb{R}^d}(L_G(u_\tau^k))^q\,\d x \geq 2^{1-2q}(\KK (u_\tau^k))^{2q-1}.
\end{equation}
Thanks to \eqref{scalarestimate}, \eqref{scalar2} and \eqref{scalar3} we find
\[
\langle u_\tau^k,L_G(u_\tau^k)\rangle_{1-s} \geq \tilde C_0 (\KK(u_\tau^k))^{2q-1},
\]
where {\RRR $\tilde C_0:=2^{1-2q}A_{d,s}=2^{-\frac{3d+4(1-s)}{d}}A_{d,s}$.}
By applying Lemma \ref{lemma:decay1} we obtain
\begin{equation}\label{newE}
\KK(u_\tau^{k})+\tilde C_0\tau(\KK(u_\tau^k))^{\beta_0}\le \KK(u_\tau^{k-1}),\qquad k=1,2,\ldots,
\end{equation}
%
%
%$$g'(t) \leq -C(\KK(u_t))^{2q-1} - \eps \|u_t\|^2_{\dot H^{1-s}(\R^d)},$$
%for $C=\tilde C/C_{1-s}^2$.
%
%By Lagrange's mean value Theorem, for every $t>0$ there exists $\theta(t)\in (0,t)$ such that
%\[
%\frac{\FF_s(u_\tau^k) -\FF_s(S_t(u_\tau^k))}{t} = g'(\theta(t)) \leq -C(\KK(S_{\theta(t)}(u_\tau^k))^{2q-1} - \eps\|S_{\theta(t)}(u_\tau^k)\|^2_{\dot H^{1-s}(\R^d)}.
%\]
%By a lower semicontinuity argument it follows that
%\[
%C(\KK(u_\tau^k)^{2q-1} +\eps \|u_\tau^k\|^2_{\dot H^{1-s}(\R^d)} \leq   D_\VV \FF_s(u_\tau^k).
%\]
%Using Proposition \ref{prop:FI} and passing to the limit as $\eps\to 0$
%we finally obtain
%\begin{equation}\label{E1}
%\tau C( \KK(u_\tau^k))^\beta +\KK(u_\tau^k)  \leq \KK(u_\tau^{k-1}), 
%\end{equation}
where $\beta_0:=2q-1=\frac{3d+4(1-s)}{d}$.

Let us now consider, for $p\in (1,+\infty)$ the case $\GG=\GG_p$, with density function $G=G_p$.
Taking into account that $L_{G_p}(u)=u^p$, by  Lemma \ref{lemma:decay1} and the Stroock-Varopoulos inequality (Proposition \ref{prop:BB}), 
we have $(u_\tau^k)^{(p+1)/2}\in \dot H^{1-s}(\mathbb{R}^d)$ and
%$$g'(t) \leq -\frac{4p}{(p+1)^2} \|u_t^{(p+1)/2}\|^2_{\dot H^{1-s}(\R^d)} -\eps \|u_t\|^2_{\dot H^{1-s}(\R^d)}.$$
%Using Proposition \ref{prop:FI} and passing to the limit as $\eps\to 0$
%we finally obtain
%Using the same argument of Lemma \ref{lemma:ed} we obtain that
%$$ \tau \frac{4p}{(p+1)^2} \|(u_\tau^k)^{(p+1)/2}\|^2_{\dot H^{1-s}} +\tau\eps \|u_\tau^k\|^2_{\dot H^{1-s}}
%+ \frac1{p-1}\|u_\tau^k\|^p_{L^p(\R^d)} +\eps\HH(u_\tau^k) \leq \frac1{p-1}\|u_\tau^{k-1}\|^p_{L^p(\R^d)} +\eps\HH(u_\tau^{k-1}). $$
%Letting $\eps\to 0$ we obtain
$$ \tau \frac{4p(p-1)}{(p+1)^2} \|(u_\tau^k)^{(p+1)/2}\|^2_{\dot H^{1-s}(\R^d)}
+\|u_\tau^k\|^p_{L^p(\R^d)}  \leq \|u_\tau^{k-1}\|^p_{L^p(\R^d)}\qquad k=1,2,\ldots. $$
{\RRR By \eqref{GNS} with $r=1-s$ in the case $d\ge 2$ and \eqref{extrasobolev1} in the case $d=1$,
both with the choice $u=u_\tau^k$,} we obtain
\begin{equation}\label{E2}
\tau \frac{4p(p-1)}{(p+1)^2}B_{d,s}( \|u_\tau^k\|^p_{L^p(\R^d)})^{\beta_p}
+\|u_\tau^k\|^p_{L^p(\R^d)}  \leq \|u_\tau^{k-1}\|^p_{L^p(\R^d)}, \qquad k=1,2,\ldots,
\end{equation}
where $\beta_p:=\frac{pd+2(1-s)}{(p-1)d}$,  {\RRR $B_{d,s}:=S_{d,1-s}^{-2}$ if $d\ge 2$ and $B_{1,s}:=S_{1,\frac{1-s}{4-2s}}^{4s-8}$.}

Now we are ready to conclude for both the cases $\GG=\KK$ and $\GG=\GG_p$.
Setting $a_k:=\KK(u_\tau^k)$ in the first case and
$a_k:=\|u_\tau^k\|^p_{L^p(\R^d)}$ in the second case, the relations \eqref{newE}, \eqref{E2} read
$$a_k - a_{k-1}\leq -\tau C a_k^\beta,$$
where $C=\tilde C_0$, $\beta=\beta_0$  in the first case and
{\RRR $C=\tilde C_p:=\frac{4p(p-1)}{(p+1)^2}B_{d,s}$,} $\beta=\beta_p$ in the second case. In both cases,
we apply Proposition \ref{prop:DI} and Proposition  \ref{prop:CL} with the choice $\phi(a)=\frac{C}{\beta+1}a^{\beta+1}$.
The solution of the Cauchy problem \eqref{CauchyPb} is then $b(t)=(b_0+C(\beta-1)t)^{1/(1-\beta)}$.
Since $\beta>1$, the function $y\mapsto y^{1/(1-\beta)}$ is decreasing in $(0,+\infty)$.
Consequently we have $b(t)\leq \min\{b_0,(C(\beta-1)t)^{1/(1-\beta)}\}$.
Finally $$a_k\le b_k\leq b(k\tau)+|b_k-b(k\tau)|\leq b(k\tau) + \frac{1}{\sqrt{2}}\phi'(b_0)\tau.$$
With the choice  $b_0=\KK(u_\tau^0)$  in the first case, we obtain \eqref{discreteREK}. With the choice
 $b_0=\|u_\tau^0\|^p_{L^p(\R^d)}$  in the second case,    
 we  obtain \eqref{discreteRELp}.
 
% {\RRR
% Let us conclude by considering the case $d=1$, $0<s<1/2$. In such case we can not use \eqref{fractionalembedding}, \eqref{FractionalGNS} and \eqref{GNS} with $r=1-s$, since $1-s>1/2$. Instead, we take advantage of the interpolated inequalities given in Remark \ref{rem:oned}.
%We use inequality \eqref{eq:Bs} in place of \eqref{FractionalGNS} to obtain \eqref{scalar2} with a different multiplicative constant and with $q=4-2s$. %Therefore  we still get \eqref{discreteREK}. %with $\tilde C_0:=2^{-3-4(1-s)}S_{1,{(1-s)}/{(4-2s)}}^{2s-2}$.
%Similarly, to obtain \eqref{E2} we use inequality \eqref{extrasobolev1} in place of \eqref{GNS}. \eqref{discreteREK} and \eqref{discreteRELp} follow.
%%Consequently, \eqref{discreteRELp}  holds.
% 
 %}
%\end{proof}

{\RRR 
}
\end{proof}

We may now pass to the limit as $\tau\to 0$  and prove the decay estimates for the solution.
\begin{theorem}\label{theop}  Let $\{u_\tau^k:k=0,1,2,\ldots\}$ be the sequence given by  {\rm Proposition \ref{prop:existenceMM}}.
%and $u_\tau$  the  discrete solution, see \eqref{PC}.
If $u\in AC^2([0,+\infty);(\PP_2({\R}^d),W))$ is a corresponding limit curve given by {\rm Theorem \ref{th:convergence1}}, then
%there exists a constant $C_0$ depending only on $s$ and $d$ such that
\begin{equation*}\label{REK}
    \KK(u(t)) \leq C_0 t^{-\gamma_0},\qquad t>0,
\end{equation*}
where $C_0,\gamma_0$ are positive constants, whose explicit value is found in {\rm Lemma \ref{lemma:decay2}}, %$\gamma_0=\frac{1}{2}\frac{d}{d+2(1-s)}$ and  $\beta_0=\frac{3d+4(1-s)}{d}$
and
\begin{equation*}\label{REK1}
    \KK(u(t)) \leq \lim_{\tau\to 0}\KK(u_\tau^0)\qquad t>0.
\end{equation*}
Moreover,
for every $p\in(1,+\infty)$ there holds % exists a constant $C_p$ depending only on $s$, $d$ and $p$ such that
\begin{equation*}\label{RELp}
    \|u(t)\|_{L^p(\R^d)} \leq C_pt^{-\gamma_p},\qquad t>0,
\end{equation*}
where the positive constants $C_p,\gamma_p$ are found in {\rm Lemma \ref{lemma:decay2}} as well, %$\gamma_p=\frac{p-1}{p}\frac{d}{d+2(1-s)}$ and  $\beta_p=\frac{pd+2(1-s)}{(p-1)d}>1$, 
and
\begin{equation*}\label{RELp1}
    \|u(t)\|_{L^p(\R^d)} \leq \lim_{\tau\to 0}\|u_\tau^0\|_{L^p(\R^d)},\qquad t>0.
\end{equation*}
\end{theorem}

\begin{proof}
With the choice of $u_\tau^0$ from Section \ref{subsection:MM} we immediately have that
$$ \lim_{\tau\to 0} \tau (\KK(u_\tau^0))^{\beta_0} = 0, \qquad  \lim_{\tau\to 0} \tau^{1/p} \|u_\tau^0\|^{\beta_p}_{L^p(\R^d)} =0,$$
since the $L^p$ norms of $u_\tau^0$ diverge at most logarithmically as $\tau\to 0$.
The proof is now a consequence of Lemma \ref{lemma:decay2}, of the narrow convergence \eqref{narrowconv}
and of the lower semi continuity of $\KK$ and of the $L^p$ norms with respect to the narrow convergence.
\end{proof}

\section{Euler-Lagrange equation for the minimizers}

Thanks to Lemma \ref{lemma:ed}, we have enough regularity to obtain an Euler-Lagrange equation for discrete minimizers.
This necessary condition \eqref{euler} on the minimizers of the scheme is the first step towards
a discrete version of a weak formulation of the equation \eqref{equation}, (see \eqref{euler2}).

\begin{lemma}\label{lemma:euler}
Let $u_0\in D(\FF_s)$. Let $\{u_{\tau}^k:k=0,1,2,\ldots\}$ be the solution sequence to \eqref{minmov} given by {\rm Proposition \ref{prop:existenceMM}}
and $v_{\tau}^k:=K_s * u_{\tau}^k$.
Then, for any integer $k\geq 1$ there holds
\begin{equation}\label{euler}
	\int_{\R^d}\nabla v_{\tau}^k\cdot\eta \, u_{\tau}^k\,\d x=\frac1\tau \int_{\R^d}(T_{u_\tau^k}^{u_\tau^{k-1}} -\id)\cdot\eta \, u_{\tau}^k\,\d x,
	\qquad \forall \, \eta \in C^{\infty}_c(\mathbb{R}^d;\mathbb{R}^d),
\end{equation}
where  $T_{u_\tau^k}^{u_\tau^{k-1}}$ is the optimal transport map from $u_\tau^k$ to $u_\tau^{k-1}$ and
%$\mathbf{T}_{\tau}^k$ denotes the optimal transport map between $u_{\tau}^k$ and $u_{\tau}^{k-1}$ and 
$\id$ is the identity map on $\mathbb{R}^d$.
Moreover, there holds
\begin{equation}\label{ab}
	\int_{\R^d}|\nabla v_\tau^k|^2u_\tau^k\,dx = \frac{1}{\tau^2}W^2(u_\tau^k,u_\tau^{k-1}), \quad k=1,2,3,...
\end{equation}

\end{lemma}
\begin{proof}
Let $\eta\in C^{\infty}_c(\mathbb{R}^d;\mathbb{R}^d)$. For $\delta \geq 0$ we define $\Phi_\delta:\R^d\to\R^d$
by $\Phi_\delta(x)=x+\delta\eta(x)$.
Clearly there exists $\delta_0>0$ such that
\begin{equation*}
	\frac12 \leq \det(\nabla \Phi_\delta(x)) \leq \frac32 \qquad \forall x \in \R^d , \quad \forall \delta \in [0,\delta_0]
\end{equation*}
and $\Phi_\delta$ is a global diffeomorphism.
In this proof, for simplicity, we use the following notation:
$u:=u_\tau^k$ and $u_\delta:=(\Phi_\delta)_\#u$.

By the minimum problem \eqref{minmov} we have, for $\delta>0$,
\begin{equation}\label{asnote}
0 \le \frac1\delta \left( \FF_s(u_\delta)-\FF_s(u)\right)
+ \frac1\delta \left(\frac{1}{2\tau}W^2(u_\delta ,u_\tau^{k-1})-\frac{1}{2\tau}W^2(u,u_\tau^{k-1})\right)
\end{equation}
A standard  computation entails
\begin{equation}\label{standard}
 \lim_{\delta \to 0}\frac1\delta \left(\frac{1}{2\tau}W^2(u_\delta ,u_\tau^{k-1})-\frac{1}{2\tau}W^2(u,u_\tau^{k-1})\right)
 = - \frac{1}{\tau}\int_{\R^d} (T_\tau^k-\id)\cdot \eta u\,\d x.
\end{equation}
We have to compute
\begin{equation}\label{nonstandard}
 \lim_{\delta \to 0}\frac1\delta \left( \FF_s(u_\delta)-\FF_s(u)\right).
\end{equation}
Since for $a,b\in\C$ it holds
$|a|^2-|b|^2=(\bar a+\bar b)(a-b)+\bar ab-\bar ba$
and $\bar{\hat u}(\xi)=\hat u(-\xi)$,
we obtain
$$ (2\pi)^d\left( \FF_s(u_\delta)-\FF_s(u)\right) =
\frac12\int_{\R^d}|\xi|^{-2s}\big(\hat u_\delta(-\xi)+\hat u(-\xi)\big)\big(\hat u_\delta(\xi)-\hat u(\xi)\big)\,\d\xi $$
because
$$ \int_{\R^d}|\xi|^{-2s}\hat u_\delta(-\xi)\hat u(\xi)\,\d\xi
= \int_{\R^d}|\xi|^{-2s}\hat u_\delta(\xi)\hat u(-\xi)\,\d\xi .  $$
After defining $\hat v_\delta(\xi)=|\xi|^{-2s}\hat u_\delta(\xi)$  and $\hat v(\xi)=|\xi|^{-2s}\hat u(\xi)$ we write
\begin{equation}\label{elv}
 \begin{aligned}  \frac{(2\pi)^d}{\delta}\left( \FF_s(u_\delta)-\FF_s(u)\right) &=
\frac12\int_{\R^d}\big(\hat v_\delta(-\xi)+\hat v(-\xi)\big)\frac1\delta\big(\hat u_\delta(\xi)-\hat u(\xi)\big)\,\d\xi \\
&=\frac12\int_{\R^d}|\xi|\big(\hat v_\delta(-\xi)+\hat v(-\xi)\big)|\xi|^{-1}\frac1\delta\big(\hat u_\delta(\xi)-\hat u(\xi)\big)\,\d\xi.
\end{aligned}
\end{equation}
We show that $|\xi|\hat v_\delta(-\xi)$ converges to $|\xi|\hat v(-\xi)$ strongly in $L^2(\R^d)$ as $\delta \to 0$.
First of all we observe that there exists a constant $c$ such that
$$ \|u_\delta\|_{H^{1-s}(\R^d)}\leq c, \qquad \forall \delta\in [0,\delta_0].$$
In order to obtain this bound we write
$u_\delta= \phi_\delta u\circ\Phi_\delta^{-1} + u\circ\Phi_\delta^{-1}$, where $\phi_\delta=\det\nabla\Phi_\delta^{-1}-1$. Since $\Phi_\delta^{-1}$ is a global diffeomorphism, close to the identity, and clearly there exists a constant $\tilde c>0$ such that $|\Phi_\delta(x)-\Phi_\delta(y)|\ge \tilde c|x-y|$ for any $x,y\in\mathbb{R}^d$ and any $\delta \in [0,\delta_0]$, we get $\|u\circ\Phi_\delta^{-1}\|_{H^{1-s}(\mathbb{R}^d)}\le \tilde c\|u\|_{H^{1-s}(\mathbb{R}^d)}$, see \cite[Corollary 1.60]{BCD}.
A similar estimate holds true as well if we multiply by the smooth compactly supported function $\phi_\delta$, see also of \cite[Theorem 1.62]{BCD}.
Then $$\|u_\delta-u\|_{H^{1-s}(\R^d)}\leq c+\|u\|_{H^{1-s}(\mathbb{R}^d)}, \qquad \forall \delta\in [0,\delta_0].$$
Since $\supp(u_\delta-u)=\supp\eta$ is compact we have that
$\{u_\delta-u\}_{\delta\in[0,\delta_0]}$ is strongly compact in $H^r(\R^d)$ for any $r<1-s$.
Since $u_\delta \to u$ narrowly as $\delta\to 0$ we obtain that $\|u_\delta-u\|_{H^{r}(\R^d)}\to 0$ as $\delta\to 0$.

Since $-s<1-2s<1-s$, choosing $r\in (1-2s,1-s)\cap(0,1-s)$,
by interpolation we have
$$ \|\nabla v_\delta-\nabla v\|_{L^2(\R^d)} = \|u_\delta-u\|_{\dot H^{1-2s}(\R^d)}\leq \|u_\delta-u\|^{1-\theta}_{\dot H^{-s}(\R^d)}\|u_\delta-u\|^\theta_{\dot H^{r}(\R^d)},
$$
where $1-2s=(1-\theta)(-s)+\theta r $.
Since $ \|u_\delta-u\|_{\dot H^{-s}(\R^d)}$ is uniformly bounded for $\delta \in (0,\delta_0)$ we obtain the strong convergence
in $L^2(\R^d)$ of $|\xi| \hat v_\delta(-\xi)$ to $|\xi|\hat v(-\xi)$.

For every $\xi\in\R^d$ the function $g_\xi:[0,+\infty)\to\R$ defined by $g_\xi(\delta)= \hat u_\delta(\xi)$ is of class $C^1$
and $$ g'_\xi(\delta)= -i\xi\cdot\int e^{-i\xi\cdot(x+\delta\eta(x))}\eta(x)u(x)\,\d x. $$
The continuity of the derivative follows from its expression and dominated convergence Theorem.
Indeed, by definition of image measure, i.e.,
$$ \hat u_\delta(\xi)=\int_{\R^d}e^{-i\xi\cdot(x+\delta\eta(x))}u(x)\,\d x,$$
we have
$$
\begin{aligned}
\frac1h\big(  \hat u_{\delta+h}(\xi)- \hat u_\delta(\xi)\big) &=\int_{\R^d}\frac1h(e^{-i\xi\cdot(h\eta(x))}-1)  e^{-i\xi\cdot(x+\delta\eta(x))}u(x)\,\d x
 \\
&\to -i\xi\cdot\int e^{-i\xi\cdot(x+\delta\eta(x))}\eta(x)u(x)\,\d x
\end{aligned}
$$
as $h\to 0$,
by dominated convergence Theorem.

By Lagrange Theorem for every $\xi$ and $\delta>0$ there exist $\delta_\xi\in [0,\delta)$ such that
$\frac1\delta\big(\hat u_\delta(\xi)-\hat u(\xi)\big)= g'_\xi(\delta_\xi)$.
Since $|g'_\xi(\delta_\xi)|\leq |\xi|\|\eta\|_{L^\infty}$
we obtain that
$|\xi|^{-1}\frac1\delta\big(\hat u_\delta(\xi)-\hat u(\xi)\big)$ converges to $-i |\xi|^{-1}\xi \cdot (\widehat{\eta u})(\xi)$ in the sense of distributions.
But
$$ g'_\xi(\delta)= -i\xi\cdot\widehat{(\eta u)_\delta}(\xi)$$
where $(\eta u)_\delta=(\Phi_\delta)_\#(\eta u)$,
and $\|(\eta u)_\delta \|_{L^2(\R^d)} \leq 2\|\eta u\|_{L^2(\R^d)}$,
so that $|\xi|^{-1}\frac1\delta\big(\hat u_\delta(\xi)-\hat u(\xi)\big)$ is bounded in $L^2(\R^d)$.
Consequently
$|\xi|^{-1}\frac1\delta\big(\hat u_\delta(\xi)-\hat u(\xi)\big)$ converges to $-i |\xi|^{-1}\xi \cdot (\widehat{\eta u})(\xi)$ weakly in $L^2(\R^d)$ as well.

Eventually, we may pass to the limit in \eqref{elv} by strong vs weak convergence, and using Plancherel Theorem we obtain
\begin{equation}\label{ending}
\begin{aligned}
	(2\pi)^d \lim_{\delta \to 0}\frac1\delta \left( \FF_s(u_\delta)-\FF_s(u)\right) &= -i\int_{\R^d} \hat v(-\xi)\xi\cdot (\widehat{\eta u})(\xi)\, \d\xi \\
	& = -i\sum_{j=1}^d \int_{\R^d} \hat v(-\xi)\xi_j\cdot (\widehat{\eta_j u})(\xi)\, \d\xi	\\
	&= (2\pi)^d \sum_{j=1}^d \int_{\R^d} \de_{x_j} v(x)\eta_j(x) u(x)\, \d x	\\
	&=(2\pi)^d \int_{\R^d} \nabla v(x)\cdot \eta(x) u(x)\, \d x.	
\end{aligned}
\end{equation}
In conclusion, by combining \eqref{asnote}, \eqref{standard}, \eqref{nonstandard}, \eqref{elv} and \eqref{ending}, we get
\begin{equation*}
0 \le \int_{\R^d} \nabla v\cdot \eta u\d x - \frac{1}{\tau}\int_{\R^d} (T_\tau^k-\id)\cdot \eta u\d x.
\end{equation*}
The above inequality is valid also for $-\eta$ instead of $\eta$, so that it is indeed an equality and \eqref{euler} holds. From \eqref{euler}, it follows that
$\tau u_\tau^k\nabla v_\tau^k=(T_{u_\tau^k}^{u_\tau^{k-1}}-I)u_\tau^k$ holds a.e. in $\mathbb{R}^d$. 
Since $W^2(u_\tau^k,u_\tau^{k-1})=\int_{\mathbb{R}^d}|T_{u_\tau^k}^{u_\tau^{k-1}}-I|^2\,u_\tau^k\,\d x$,
\eqref{ab} follows as well.
\end{proof}

%The Euler-Lagrange equation \eqref{regularizedeuler} extends to an equality in $L^1(\mathbb{R}^d)$. In fact, since $u_{\tau,\eps}\in\PP_2(\mathbb{R}^d)$ solves \eqref{regularizedminmov}, considering the right hand side therein we have
%\[
%\int_{\mathbb{R}^d}|(\mathbf{T}_{\tau,\eps}^0-\id)u_{\tau_\eps}|\le \int_{\mathbb{R}^d}|\mathbf{T}_{\tau,\eps}^0-\id|^2 u_{\tau_\eps}=W_2^2(u_{\tau,\eps},u^0)<+\infty.
%\]
%Moreover, since we proved that $u_{\tau,\eps}\in  H^{1-s}(\mathbb{R}^d)$, we have (DA VERIFICARE) $\nabla p_{\tau,\eps}\in \dot H^s(\mathbb{R}^d)$,
%then $u_{\tau,\eps}\nabla p_{\tau,\eps}\in L^1(\mathbb{R}^d)$. We conclude from \eqref{regularizedeuler} that $u_{\tau,\eps}\in W^{1,1}(\mathbb{R}^d)$.

\section{Convergence and energy dissipation}

In this Section we prove that the limit curve obtained by means of Theorem \ref{th:convergence1} is indeed a 
gradient flow solution to problem \eqref{equation}:
it satisfies \eqref{equation} in the sense of distributions and a corresponding energy dissipation inequality holds.

\subsection{Convergence}

\begin{lemma}\label{lemma:convergence2}
Let $u_0\in \dot H^{-s}({\R}^d)\cap\PP_2({\R}^d)$, $u_\tau$ the piecewise constant curve defined in \eqref{PC} and
$v_\tau(t):=K_{s}\ast u_\tau(t)$ defined for $t\geq 0$.
Given a vanishing sequence $\tau_n$, let
$u_{\tau_{n}}$ be a narrowly convergent subsequence (not relabeled) given by {\rm Theorem \ref{th:convergence1}}, 
$u$ its limit curve and  $v(t):=K_{s}\ast u(t)$ for $t\geq 0$.

Then, for any $T_0>0$ and $T>T_0$ we have
$u\in L^2((T_0,T); H^{1-s}(\R^d))$ and $\nabla v \in L^2((T_0,T);L^2({\R}^d))$.
Moreover the following convergences hold:
\[
 \phi u_{\tau_{n}}\to \phi u \quad \mbox{strongly in } L^2((T_0,T);H^r({\R}^d)) \quad \mbox{ as $n\to\infty$, } \forall \phi \in \SS(\R^d), \forall r< 1-s,
\]
\[
 u_{\tau_{n}}\to u \quad \mbox{strongly in } L^2((T_0,T);L^2_{loc}({\R}^d)) \quad \mbox{ as $n\to\infty$,}
\]
\begin{equation}\label{convergenceVF}
 \nabla v_{\tau_{n}}\to \nabla v \quad \mbox{weakly in } L^2((T_0,T);L^2({\R}^d)) \quad \mbox{ as $n\to\infty$.}
\end{equation}
If, in addition, $u_0\in D(\HH)$, then the above results also hold for $T_0=0$.
\end{lemma}

\begin{proof}
Let $T_0>0$.  By the definition of $u_\tau^0$ we have that the error in \eqref{discreteREK} vanishes as $\tau\to 0$, 
i.e., $\lim_{\tau\to 0}\tau (\KK(u_\tau^0))^{\beta_0}=0$.
As in Corollary \ref{furthercorollary} we let $N_0(\tau)=[T_0/\tau]$. By \eqref{discreteREK} 
%({\RRR when $d=1$ and $s\in (0,1/2)$ we use the variant of \eqref{discreteREK} discussed in Remark \ref{rem:oned}})
 and the inequality $\HH(u) \leq \KK(u)$ we obtain that
\begin{equation}\label{boundentropy}
	\limsup_{\tau\to 0} \HH(u_\tau^{N_0(\tau)})\leq C_0T_0^{-\gamma_0},
\end{equation}
where the value of the constants $C_0$ and $\gamma_0$ is stated in Lemma \ref{lemma:decay2}.
Since by interpolation, for $\theta=s$, it holds
$$ \|u_\tau(t) \|_{L^{2}(\R^d)}\leq \|u_\tau(t) \|^{1-s}_{\dot H^{-s}(\R^d)} \|u_\tau(t) \|^s_{\dot H^{1-s}(\R^d)}, $$
then by H\"older's inequality, \eqref{basicestimate} and  \eqref{RI1} we obtain
\begin{equation}\label{L2estimate}
\begin{aligned}
& \int_{T_0}^T \|u_\tau(t) \|^2_{L^{2}(\R^d)}\,\d t 
\leq \Big( \int_{T_0}^T \|u_\tau(t) \|^2_{\dot H^{-s}(\R^d)}\,\d t  \Big)^{1-s} \Big( \int_{T_0}^T \|u_\tau(t) \|^2_{\dot H^{1-s}(\R^d)}\,\d t  \Big)^{s} \\
&\qquad \leq \Big( 2\FF_s(u_0) (T-T_0)  \Big)^{1-s} \Big( \int_{T_0}^T \|u_\tau(t) \|^2_{\dot H^{1-s}(\R^d)}\,\d t  \Big)^{s} \\
&\qquad\leq \Big( 2\FF_s(u_0) (T-T_0)  \Big)^{1-s} \Big(    \HH(u_\tau^{N_0(\tau)}) + c\Big(1 +T\FF_s(u_0)+\int_{\R^d}|x|^2\,\d u_0(x) \Big)\Big)^{s}.
\end{aligned}
\end{equation}
From  \eqref{RI1}  and the last estimate, by lower semicontinuity we obtain that $u\in L^2((T_0,T); H^{1-s}(\R^d))$.

Taking into account that $-s<1-2s<1-s$, by interpolation we obtain, for  $\theta=1-s$,
$$ \|u_\tau(t) \|_{\dot H^{1-2s}(\R^d)}\leq \|u_\tau(t) \|^{s}_{\dot H^{-s}(\R^d)} \|u_\tau(t) \|^{1-s}_{\dot H^{1-s}(\R^d)}, $$
then by Holder's inequality, \eqref{basicestimate} and  \eqref{RI1} we obtain as above
\begin{equation}\label{vestimate}
\begin{aligned}
& \int_{T_0}^T \|u_\tau(t) \|^2_{\dot H^{1-2s}(\R^d)}\,\d t
\leq \Big( \int_{T_0}^T \|u_\tau(t) \|^2_{\dot H^{-s}(\R^d)}\,\d t  \Big)^{s} \Big( \int_{T_0}^T \|u_\tau(t) \|^2_{\dot H^{1-s}(\R^d)}\,\d t  \Big)^{1-s} \\
&\qquad \leq \Big( 2\FF_s(u_0) (T-T_0)  \Big)^{s} \Big( \HH(u_\tau^{N_0(\tau)}) + c\Big(1 +T\FF_s(u_0)+\int_{\R^d}|x|^2\,\d u_0(x) \Big)\Big)^{1-s}.
\end{aligned}
\end{equation}
Since  $\widehat{v_\tau(t)}(\xi)=|\xi|^{-2s}\widehat{u_\tau(t)}(\xi)$, by Plancherel Theorem we have
$ \|\nabla v_\tau(t)\|_{L^{2}(\R^d)} =  \|u_\tau(t) \|_{\dot H^{1-2s}(\R^d)}$.
From the previous estimate it follows that
$\{\nabla v_\tau\}_{\tau>0}$ is weakly compact in $L^2((T_0,T);L^2({\R}^d))$.
%{\color{blue}
Moreover $\nabla v_{\tau_k}$ converges to  $\nabla v$ in the sense of distributions in $\R^d\times(T_0,T)$.
Indeed for $\varphi\in C^\infty_c(\R^d\times(T_0,T);\R^d)$, denoting by $\varphi_t$ the function $x\mapsto \varphi(x,t)$, by Plancherel's Theorem 
we have 
$$
(2\pi)^d\int_{T_0}^T\int_{\R^d} \nabla v_{\tau_k} \cdot \varphi \,\d x\, \d t 
= -i\int_{T_0}^T\int_{\R^d} |\xi|^{-2s} \widehat{u_{\tau_k}(t)}(-\xi) \xi \cdot \widehat\varphi_t(\xi) \,\d\xi\, \d t .
$$
Since  $| |\xi|^{-2s} \widehat{u_{\tau_k}(t)}(-\xi) \xi \cdot \widehat\varphi_t(\xi)| 
\leq  |\xi|^{1-2s} | \widehat\varphi_t(\xi)| $ and $\widehat\varphi_t \in \SS(\R^d)$ for every $t\in(T_0,T)$, by \eqref{narrowconv} and
Lebesgue dominated convergence the right hand side of the above formula converges to
$$
 -i\int_{T_0}^T\int_{\R^d} |\xi|^{-2s} \widehat{u(t)}(-\xi) \xi \cdot \widehat\varphi_t(\xi) \,\d\xi\, \d t
= (2\pi)^d\int_{T_0}^T\int_{\R^d} \nabla v \cdot \varphi \,\d x\, \d t .
$$
For the stated compactness in  $L^2((T_0,T);L^2({\R}^d))$ we obtain \eqref{convergenceVF}.
%}

Let $\phi\in\SS(\R^d)$, $r\in [0,1-s)$ and $\eps>0$.
Since $ \|u_\tau(t)\|^2_{H^{-s}(\R^d)} \leq  \|u_\tau(t)\|^2_{\dot H^{-s}(\R^d)}\leq 2\FF_s(u_0)$,
then $\{\phi u_\tau(t)\}_{\tau>0}$ is compact in $H^{-s-\eps}(\R^d)$ for any $t$.
%{\color{blue}
Thus, for any $t$ we can select a subsequence $\tau_{{n}_{k(t)}}$ of $\tau_n$ such that 
$\phi u_{\tau_{{n}_{k(t)}}}(t) \to { w_t}$ strongly for some $w_t\in H^{-s-\eps}(\R^d)$.  
Actually, the subsequence is shown not to depend on $t$ thanks to \eqref{narrowconv} and the uniqueness
of the limit. As a result, we have that $\phi u_{\tau_{n}}(t) \to \phi u(t)$ in $H^{-s-\eps}(\R^d)$ for any $t>0$ and 
for any $\phi\in \SS(\R^d)$.
%}
%It follows that $\|\phi u_{\tau_{n}}(t) -\phi u(t)\|^2_{H^{-s-\eps}(\R^d)}\to 0$ as $n\to+\infty$ for any $t$.
By Proposition \ref{prop:Sobolev} there exists a constant $C$ such that
$$
\begin{aligned}
 \|\phi u_\tau(t) -\phi u(t)\|^2_{H^{-s-\eps}(\R^d)} &\leq C\|u_\tau(t)-u(t)\|^2_{H^{-s-\eps}(\R^d)} \\
 & \leq C\|u_\tau(t)-u(t)\|^2_{H^{-s}(\R^d)} \\
 &\leq 2C \|u_\tau(t)\|^2_{H^{-s}(\R^d)} +2C\|u(t)\|^2_{H^{-s}(\R^d)} \leq 8C \FF_s(u_0).
\end{aligned}
 $$
Then by dominated convergence we have that
$\int_{T_0}^T \|\phi u_{\tau_n}(t) -\phi u(t)\|^2_{H^{-s-\eps}(\R^d)}\,\d t \to 0 $ as $n\to+\infty$.
For $\theta= (r+s+\eps)/(1+\eps)$, by interpolation we have
$$
\begin{aligned}
 & \int_{T_0}^T \|\phi u_\tau(t) -\phi u(t)\|^2_{H^{r}(\R^d)}\,\d t  \\
& \qquad\le\Big(\int_{T_0}^T \|\phi u_\tau(t) -\phi u(t)\|^2_{H^{-s-\eps}(\R^d)}\,\d t\Big)^{1-\theta} \Big(\int_{T_0}^T \|\phi u_\tau(t) -\phi u(t)\|^2_{H^{1-s}(\R^d)}\,\d t\Big)^\theta.
\end{aligned}
$$
Since by Proposition \ref{prop:Sobolev} there holds
\[\begin{aligned}\int_{T_0}^T \|\phi u_\tau(t) -\phi u(t)\|^2_{H^{1-s}(\R^d)}\,\d t &\leq C \int_{T_0}^T \|u_\tau(t) -u(t)\|^2_{H^{1-s}(\R^d)}\,\d t \\&\leq 
4C \Big( \HH(u_\tau^{N_0(\tau)}) + c\Big(1 +T\FF_s(u_0)+\int_{\R^d}|x|^2\,\d u_0(x) \Big)  \end{aligned}\]
the first convergence result follows.

In order to show the second convergence result
let $K\subset \R^d$ be a compact and we choose $\phi:\R^d\to\R$ such that $\phi\in C_c^\infty(\R^d)$,
$0\leq \phi \leq 1$, $\phi=1$ on $K$ and $r=0$.
Since
$ \|u_\tau(t) -u(t)\|^2_{L^{2}(K)}\leq \|\phi u_\tau(t) -\phi u(t)\|^2_{L^{2}(\R^d)}$ we conclude.

If $T_0=0$ then $N_0(\tau)=0$, and the last assertion follows from the previous estimates taking into account that $\HH(u_\tau^0)\leq\HH(u_0)$.
\end{proof}

\begin{theorem}
\label{th:weak_form}
If $u\in AC^2([0,+\infty);(\PP_2({\R}^d),W))$ is a limit curve given by {\rm Theorem \ref{th:convergence1}}, 
and $v(t):=K_{s}\ast u(t)$ for $t\geq 0$, then $u$ satisfies the equation in \eqref{equation} in the following weak form
\[
\int_0^{+\infty}\int_{\R^d} ( \partial_t\varphi - \nabla\varphi\cdot \nabla v )u\, \dd x \, \dd t=0, \quad
\text{for all }\varphi\in C^\infty_c((0,+\infty)\times\R^d).
\]
\end{theorem}

\begin{proof}
We fix $\varphi\in C^\infty_c((0,+\infty)\times\R^d)$.
By \eqref{euler} with the choice of $\eta=\nabla_x \varphi$ (depending on time) and integrating we obtain  
\begin{equation}\label{euler2}
	\int_0^{+\infty}\int_{\R^d}\nabla v_{\tau}\cdot \nabla\varphi \,u_{\tau}\,\d x\,\d t
	=\frac1\tau \int_0^{+\infty} \int_{\R^d}({T}_{\tau}-\id)\cdot \nabla\varphi \,u_{\tau}\,\d x\,\d t,
\end{equation}
where $T_\tau$ is defined as $T_\tau(t)=T_{u_\tau^k}^{u_\tau^{k-1}}$ if $t\in((k-1)\tau,k\tau]$.
By Lemma \ref{lemma:convergence2} along a suitable sequence $\tau_n$ the left hand side of \eqref{euler2} converges to
\[
\int_0^{+\infty}\int_{\R^d} \nabla\varphi\cdot \nabla v \,u\, \dd x \, \dd t
\]
By a standard argument,  the right hand side of \eqref{euler2} converges to
\[
\int_0^{+\infty}\int_{\R^d} \partial_t\varphi \, u\, \dd x \, \dd t,
\]
see for instance \cite[Theorem 11.1.6]{AGS}.
\end{proof}

\subsection{De Giorgi interpolant and Discrete energy dissipation}

In order to obtain an energy dissipation estimate we introduce the so called De Giorgi variational interpolant
(see for instance \cite[Section 3.2]{AGS}) as follows: $\tilde u_\tau(0):=u_\tau^0$ and
\begin{equation*}\begin{aligned}
    &\tilde u_\tau(t) \in  {\rm Argmin}_{u\in \PP_2(\R^d)}\left\{\frac{1}{2(t-(k-1)\tau)}W^2(u,u_\tau^{k-1}) +\FF_s(u)\right\}\\
     &\text{for } t\in ((k-1)\tau,k\tau],\quad   k=1,2,\ldots
\end{aligned}\end{equation*}
 We observe that by the argument in the proof of Proposition \ref{prop:existenceMM} 
this interpolant is uniquely defined and $\tilde u_\tau(k\tau)=u_\tau^k$ for any $k\in\NN$.
\begin{proposition}
For every $t>0$, $\tilde u_\tau(t)\in H^{1-s}(\R^d)$ and, denoting by
$\tilde v_{\tau}(t):=K_s\ast \tilde u_{\tau}(t)$,
the following discrete energy identity holds for all $N\in\NN$ and $\tau>0$
\begin{equation}\label{discreteEI}
\begin{aligned}
 \frac{1}{2}\int_0^{N\tau}\int_{\R^d}\left|\nabla v_\tau \right|^2u_\tau\,\dd x\,\dd t
+ \frac{1}{2}\int_0^{N\tau}\int_{\R^d}\left|\nabla\tilde v_\tau \right|^2\tilde u_\tau\,\dd x\,\dd t
 +\FF_s(u_\tau({N\tau})) = \FF_s(u_\tau^0).
\end{aligned}
\end{equation}
Moreover 
\begin{equation}\label{DGversusPC}
    W^2(\tilde u_\tau(t),u_\tau(t))\leq 8 \tau \FF_s(u_0), \qquad \forall t\in[0,+\infty).
\end{equation}
\end{proposition}
\begin{proof}
Fixing $t>0$, by the definition of $\tilde u_\tau(t)$, the same proof of Lemma \ref{lemma:ed} shows that
$\tilde u_\tau(t)\in H^{1-s}(\R^d)$.
For $k$ such that $t\in ((k-1)\tau,k\tau]$, the same argument of Lemma \ref{lemma:euler} shows that
\begin{equation}\label{abc}
\int_{\R^d}|\nabla \tilde v_\tau(t)|^2\tilde u_\tau(t)\,\d x = \frac{1}{(t-(k-1)\tau)^2}\,W^2(\tilde u_\tau(t),u_\tau^{k-1}).
\end{equation}
From \cite[Lemma 3.2.2]{AGS} we have the one step energy identity
\begin{equation*}
\begin{aligned}
 &\frac{1}{2}\frac{W^2(u_\tau^k,u_\tau^{k-1})}{\tau}
+ \frac{1}{2}\int_{(k-1)\tau}^{k\tau}\frac{W^2(\tilde u_\tau(t),u_\tau^{k-1})}{(t-(k-1)\tau)^2}\, \d t
 +\FF_s(u_\tau^k)= \FF_s(u_\tau^{k-1}).
\end{aligned}
\end{equation*}
Defining the function $G_\tau:(0,+\infty)\to\mathbb{R}$ as
\begin{equation*}
    G_\tau(t) = \frac{W(\tilde u(t),u_\tau^{k-1})}{t-(k-1)\tau}, \qquad t\in ((k-1)\tau,k\tau],\quad k=1,2,\ldots
\end{equation*}
and summing from $k=1$ to $N$, we obtain
\begin{equation*}
 \frac{1}{2}\sum_{k=1}^N \tau\frac{W^2(u_\tau^k,u_\tau^{k-1})}{\tau^2}
+ \frac{1}{2}\int_{0}^{N\tau} G^2_\tau(t)\, \d t
 +\FF_s(u_\tau^N)= \FF_s(u_\tau^0).
\end{equation*}
Finally \ref{discreteEI} follows by \eqref{ab} and \eqref{abc}.

The estimate \eqref{DGversusPC} follows by the definition of $\tilde u(t)$, \eqref{basicestimate}, the non-negativity of $\FF_s$ and
the triangle inequality
(see also \cite[Remark 3.2.3]{AGS}).
\end{proof}

In order to pass to the limit by lower semicontinuity in \eqref{discreteEI} 
we recall the following result, see \cite[Theorem 5.4.4]{AGS}.
\begin{lemma}\label{lemma:VF}
If $\{\mu_n\}$ is a sequence in  $\PP(\R^d\times[0,T])$  that narrowly converges to $\mu$ and
$\{w_n\}$ is a sequence of vector fields in $\L^2(\R^d\times[0,T], \mu_n;\R^d)$
satisfying
\begin{equation}\label{boundofVF}
\sup_{n}\int_{\R^d\times[0,T]} |w_n|^2 \, \d\mu_n <+\infty,
\end{equation}
then there exists a vector field $w\in \L^2(\R^d\times[0,T],\mu;\R^d)$
and a subsequence (not relabeled here) such that
\begin{equation*}%\label{VFwc}
\lim_{n\to\infty}\int_{\R^d\times[0,T]} \varphi \cdot w_n \,\d\mu_n = \int_{\R^d\times[0,T]} \varphi \cdot w \, \d\mu, 
 \qquad \forall\,\varphi\in C^\infty_c(\R^d\times[0,T];\R^d),
\end{equation*}
and moreover
\begin{equation}\label{VFlsc}
\liminf_{n\to\infty}\int_{\R^d\times[0,T]} |w_n|^2\, \d\mu_n  \geq\int_{\R^d\times[0,T]} |w|^2 \, \d\mu .
\end{equation}
\end{lemma}

\begin{theorem}
\label{th:energy}
If $u\in AC^2([0,+\infty);(\PP_2({\R}^d),W))$ is a limit curve given by {\rm Theorem \ref{th:convergence1}}, 
and $v(t):=K_{s}\ast u(t)$ for $t\geq 0$, then $u$ satisfies the following energy dissipation inequality
\[
\FF_s(u(T))+\int_0^T\int_{\mathbb{R}^d}|\nabla v(t)|^2u(t)\,\d x\,\d t \le \FF_s(u_0), \qquad \forall \;T>0.
\]
\end{theorem}

\begin{proof}
Let $u_{\tau_n}$ be the sequence of Lemma \ref{lemma:convergence2}.   We fix $T>0$ and we apply Lemma \ref{lemma:VF} to the sequences
$\mu_n:=\frac{1}{T}u_{\tau_n}$, $w_n:=\nabla v_{\tau_n}$ and $\tilde\mu_n:=\frac{1}{T}\tilde u_{\tau_n}$, $\tilde w_n:=\nabla \tilde v_{\tau_n}$.
By  \eqref{discreteEI} with $N=N_{\tau_n}:= \lceil T/{\tau_n}\rceil $, and by \eqref{BFs}, the assumption \eqref{boundofVF} is satisfied for both the 
couples $(\mu_n,w_n)$ and   $(\tilde\mu_n,\tilde w_n)$. 
By \eqref{narrowconv} and \eqref{DGversusPC} we have that $\mu_n$ and $\tilde\mu_n$ converge narrowly to $\mu:=\frac{1}{T}u$. 
By \eqref{convergenceVF} we have that the limit point of $w_n$ and $\tilde w_n$ is the same $w=\tilde w= \nabla v$.
Since $\lim_{n\to +\infty}\tau_n N_{\tau_n}=T$, by \eqref{BFs}, the lower semi continuity of $\FF_s$,  \eqref{VFlsc} and \eqref{BFs} we conclude. 
\end{proof}

\section{Boundedness of solutions and $L^{\infty}$ decay.}\label{sec:infty}

In this section we show how to get an $L^\infty$ decay rate starting from the discrete variational approach.
We have indeed to extend the estimate of Theorem \ref{theop} to $p=\infty$. Notice that $\gamma_p$  therein 
converges as $p\to \infty$, but the constant $C_p$ blows up. Therefore, we have to go through a more refined argument.

We start by introducing  a simple recursive estimate.
\begin{proposition}
Let $Q>0, R>0$ and  $q>1$.  If a sequence of positive numbers $\{A_j\}_{\{j\geq 0\}}$ satisfies $A_j\le Q R^j A_{j-1}^q$ for every $j\geq 1$, then
\begin{equation}\label{shifted}
A_j\le  Q^{\beta(j-j_0,q)} R^{\gamma(j-j_0,q)} A_{j_0}^{q^{j-j_0}},\qquad \forall \, j>j_0\geq 0,
\end{equation}
where
\[
\beta({j,q})=\frac{q^j-1}{q-1},\qquad
\gamma({j,q})=\frac{q(q^j-1)}{(q-1)^2}-\frac j{q-1}.%=\frac{q\frac{q^j-1}{q-1}-j}{q-1}.%=\frac{q(q^j-1)}{(q-1)^2}-\frac{j}{q-1}
\]
\end{proposition}
\begin{proof}
Let $j_0=0$. By  recursively using the assumption we obtain that
\[
A_j\le \prod_{i=0}^{j-1} (Q R^{j-i})^{q^i} A_0^{q^j}= Q^{\beta(j,q)} R^{\gamma(j,q)} a_0^{q^j}, \qquad j>0,
\]
where indeed
\[
\beta({j,q})=\sum_{i=0}^{j-1}q^i=\frac{q^j-1}{q-1},
\]
\[
\gamma({j,q})=\sum_{i=0}^{j-1}(j-i)q^i=\frac{1}{q-1}\,\sum_{i=1}^j (q^i-1)=\frac{q(q^j-1)}{(q-1)^2}-\frac j{q-1}.%=\frac{q\frac{q^j-1}{q-1}-j}{q-1}.%=\frac{q(q^j-1)}{(q-1)^2}-\frac{j}{q-1}
\]
If $j_0>0$ we apply the previous formula by shifting the indexes.
\end{proof}

\begin{theorem}\label{th:infinity}
If $u\in AC^2([0,+\infty);(\PP_2({\R}^d),W))$ is a limit curve given by {\rm Theorem \ref{th:convergence1}}, then
there exists a constant $C_\infty$ depending only on $d$ and $s$ such that
\[
	\|u(t)\|_{L^\infty(\R^d)}\leq C_\infty t^{-\gamma_\infty}, \qquad  t>0,
\]
where $\gamma_\infty:=\frac{d}{d+2(1-s)}$. %and $C_\infty:=2^{(2q-1)/(q-1)}S_{d,1-s}^{2q/(3q-2)}C_2^{2(q-1)/(3q-2)}$.
\end{theorem}

\begin{proof}
%{\RRR We give the proof for the case $d\ge 2$. 
%The required modifications for the case $d=1$ and $s\in (0,1/2)$
% will be sketched in Remark \ref{rem:onedbis} below.}
%{\RRR {\bf The case $d\ge 2$.}}
Fix $t>0$ throughout.  
We let $\tau>0$ %be small enough 
and we define
\[
T_j:=t(1-2^{-j}), \quad j=0,1,2,\ldots
\]
and $j(\tau)$ as the smallest integer $j$ such that $T_j>\tau\lfloor t/\tau\rfloor$, 
where $\lfloor a\rfloor:=\max\{m\in\mathbb{Z}:m<a\}$ denotes the left continuous lower integer part of the real number $a$.
The sequence $\{T_j\}$ satisfies
\[
	\tau\lfloor t/\tau\rfloor\le T_{j(\tau)}<T_{j(\tau)+1}<T_{j(\tau)+2}<\ldots< \lim_{j\to+\infty}T_j=t,
\]
and $T_j-T_{j-1}=t2^{-j}$.
We recursively define $\tilde u_{\tau,j}$ by $\tilde u_{\tau,j(\tau)}:=u_\tau(t)$ and
\begin{equation}\label{refined}
\tilde u_{\tau,j}= {\mathrm{argmin}}_{{u\in\PP_2(\mathbb{R}^d)}} \left\{\FF_s(u)+\frac{1}{2(T_j-T_{j-1})}{W_2^2(u, \tilde u_{\tau,j-1})}\right\},\quad  j>j({\tau}) \end{equation}

For given $M>0$ we define $G(u):=(u-M)_+^2$ and $\VV$ as the  displacement convex entropy with density function $G$, according to Definition \ref{entr}. By the definition of $\tilde u_{\tau,j}$ in \eqref{refined},   Lemma \ref{lemma:decay2} can be applied and yields
\begin{equation}\label{FIM}
 (T_j-T_{j-1})\langle \tilde u_{\tau,j},L_G(\tilde u_{\tau_j})\rangle_{1-s}\le   \VV(\tilde u_{\tau,j-1}) -\VV(\tilde u_{\tau,j}),\qquad j>j(\tau).
\end{equation}
Since $L_G(u)=(u-M)^2_++2M(u-M)_+$, $u\mapsto (u-M)_+^2$ is nondecreasing  and $u\mapsto (u-M)_+$ is $1$-Lipschitz continuous, by Proposition \ref{prop:BB} we have
$$ \begin{aligned}
\langle u,L_G(u)\rangle_{1-s}&= \langle u,(u-M)^2_+\rangle_{1-s}+ 2M\langle u,(u-M)_+\rangle_{1-s} 
\\ &\geq 2M\langle u,(u-M)_+\rangle_{1-s}
\\&\geq 2M\langle(u-M)_+,(u-M)_+\rangle_{1-s}
	= 2M\|(u-M)_+\|^2_{\dot H^{1-s}(\mathbb{R}^d)}.
\end{aligned}$$
%By the same argument of the proof of Lemma \ref{lemma:decay2} it follows that 
%$$D_\VV \FF_s(\tilde u_{\tau,j})\geq 2M\|(\tilde u_{\tau,j}-M)_+\|^2_{\dot H^{1-s}(\mathbb{R}^d)}.$$
Then, since $\VV\ge 0$, from \eqref{FIM} we find
\begin{equation}\label{positiveflow}\begin{aligned}
\int_{\mathbb{R}^d}(\tilde u_{\tau,j-1}(x)-M)_+^2\,\d x
\ge 2M(T_j-T_{j-1}) \|(\tilde u_{\tau,j}-M)_+\|^2_{\dot H^{1-s}(\mathbb{R}^d)},\quad j>j(\tau).
\end{aligned}
\end{equation}

Next, we define 
\[
A_{j(\tau)}:=\|\tilde u_{\tau,j(\tau)}\|^2_{L^2(\mathbb{R}^d)}=\|u_\tau(t)\|^2_{L^2(\mathbb{R}^d)}
\]
{\RRR and we separately treat the cases $d\ge 2$ and $d=1$ in the rest of the proof.

\textbf{{The case \boldmath$d\ge 2$}}. 
%{From now on we consider the case $d\ge 2$. The modifications for the case $d=1$ will be sketched at the end of the proof.}
}
We let $q:={d}/({d-2+2s})$, so that $2q$ is the critical exponent corresponding to the Sobolev inequality \eqref{fractionalembedding} with $r=1-s$
and constant denoted by $S_{d,1-s}$.
We define the constant
$$M_\tau(t):=
\left(\frac{S_{d,1-s}^2}{t}\right)^{\frac{q}{3q-2}}\left(A_{j(\tau)}\,2^{\frac{q(3q-2)}{(q-1)^2}}\right)^{\frac{q-1}{3q-2}}=
2^{q/(q-1)}S_{d,1-s}^{2q/(3q-2)}A_{j(\tau)}^{(q-1)/(3q-2)} \, t^{-q/(3q-2)},$$
and $M_{\tau,j}:=(2-2^{-j})M_\tau(t)$ for  $j>j(\tau)$.
Finally we define 
$$A_j:=\int (\tilde u_{\tau,j}-M_{\tau,j})_+^2\,\d x,\quad j>j(\tau).$$

 Since $f-M_{\tau,j}>0$ implies $f-M_{\tau,j-1}=f-M_{\tau,j}+2^{-j}M_\tau(t)>2^{-j}M_\tau(t)>0$,  a direct computation and the Sobolev inequality \eqref{fractionalembedding} entail, for any $j>j(\tau)$
\begin{equation}
\label{eq:Agei}\begin{aligned}
A_j&\le \left(\frac{2^{j}}{M_{\tau}(t)}\right)^{2q-2}\int_{\mathbb{R}^d}(\tilde u_{\tau,j}(x)-M_{\tau,{j-1}})_+^{2q}\,\d x \\&\le
 \left(\frac{2^{j}}{M_{\tau}(t)}\right)^{2q-2}S_{d,1-s}^{2q}\,\|(\tilde u_{\tau,j}-M_{\tau,{j-1}})_+\|^{2q}_{\dot H^{1-s}(\mathbb{R}^d)}.
\end{aligned}\end{equation}
Now we make use of \eqref{positiveflow}, with $M_{\tau,j}$ in place of $M$, and %since $2M_\tau>M_{\tau,j}>M_\tau$,
we get for any $j>j(\tau)$, since $M_\tau\le M_{\tau,j}$,
\begin{equation}\label{eq:Aj1}
\begin{aligned}
A_j&\le 
\left(\frac{2^j}{M_\tau(t)}\right)^{2q-2}S^{2q}_{d,1-s}\left(\frac{2^j}{t M_\tau (t)}\right)^{q}\left(\int_{\mathbb{R}^d}(\tilde u_{\tau,j-1}(x)-M_{\tau,j-1})_+^2\,\d x\right)^{q}\\
&\le\frac{S_{d,1-s}^{2q}}{t^{q}M_\tau(t)^{3q-2}}(2^{3q-2})^j\;A_{j-1}^{q}.
\end{aligned}
\end{equation}
We may apply the recursion formula \eqref{shifted}, with $Q={S_{d,1-s}^{2q}}{t^{-q}M_\tau(t)^{2-3q}}$ and $R=2^{3q-2}$, starting from $j_0=j(\tau)$, and we get
\[\begin{aligned}
A_j&\le \left(\frac{\RRR S_{d,1-s}^{2q}}{t^{q}M_\tau(t)^{3q-2}}\right)^{\frac{q^{j-j(\tau)}-1}{q-1}}\:\left(2^{3q-2}\right)^{\frac{q(q^{j-j(\tau)}-1)}{(q-1)^2}-\frac{j-j(\tau)}{q-1}}A_{j(\tau)}^{q^{j-j(\tau)}}\\
&=\left(\frac{{\RRR S_{d,1-s}^{2q}}\,2^{q(3q-2)/(q-1)}\,A_{j(\tau)}^{q-1}}{t^qM_\tau(t)^{3q-2}}\right)^{\frac{q^{j-j(\tau)}-1}{q-1}}\: 2^{-(j-j(\tau))(3q-2)/(q-1)}\: A_{j(\tau)}\\
&=2^{-(j-j(\tau))(3q-2)/(q-1)}\: A_{j(\tau)},
\end{aligned}\]
where we have used the definition of $M_\tau$.
As $q>1$, we have $\lim_{j\to+\infty} A_j=0$.

Notice that, for $j>j(\tau)$, there holds as in Theorem \ref{th:convergence1} the basic estimate
\[
\FF_s(\tilde u_{\tau,j})+\frac{W^2(\tilde u_{\tau,j},\tilde u_{\tau,j-1})}{2(T_j-T_{j-1})}\le \FF_s(\tilde u_{\tau,j-1})\le \FF_s(u_0),
\]
so that
\[
W^2(\tilde u_{\tau,j},\tilde u_{\tau,j-1})\le{2\FF_s(u_0)}(T_j-T_{j-1})={2t\FF_s(u_0)}\,2^{-j},
\]
then
\[
W(\tilde u_{\tau,n},\tilde u_{\tau,m})\le \sqrt{2t\FF_s(u_0)}\sum_{j=m+1}^n 2^{-j/2}.
\]
Therefore, $\{\tilde u_{\tau,j}\}_{j\ge j(\tau)}$ is a Cauchy sequence, converging in $\PP_2(\mathbb{R}^d)$ as $j\to+\infty$ 
to a limit point that we denote by $\tilde u_\tau(t)$, such that
\begin{equation}\label{Cauchy}
	W(\tilde u_{\tau,m},\tilde u_\tau(t))\le \sqrt{2t\FF_s(u_0)}\sum_{j=m+1}^{+\infty}2^{-j/2}.
\end{equation}

Since $\tilde u_{\tau,j}$ narrowly converges to $\tilde u_\tau(t)$ as $j\to+\infty$, 
the lower semicontinuity of $\VV$ with respect to the narrow convergence entails 
(together with $2M_\tau(t)>M_{\tau,j}$)
\[\begin{aligned}
\int_{\mathbb{R}^d} (\tilde u_\tau(t)-2M_\tau(t))_+^2\,\d x&\le\liminf_{j\to+\infty}\int_{\mathbb{R}^d}(\tilde u_{\tau,j}-2M_{\tau}(t))_+^2\,\d x \\&\le\liminf_{j\to+\infty}\int_{\mathbb{R}^d} (\tilde u_{\tau,j}-M_{\tau,j})_+^2\,\d x=\lim_{j\to+\infty}A_j=0,
\end{aligned}\]
that is
\begin{equation}\label{finalestimate}
\|\tilde u_\tau(t)\|_{L^\infty(\mathbb{R}^d)}\le 2M_\tau(t) =2^{(2q-1)/(q-1)}S_{d,1-s}^{2q/(3q-2)}A_{j(\tau)}^{(q-1)/(3q-2)} \, t^{-q/(3q-2)}.
\end{equation}
However, we apply the estimate \eqref{discreteRELp} for $p=2$ to see that
\[
A_{j(\tau)}=
%\int_{\mathbb{R}^d}(u_\tau(t)-M_{\tau,j(\tau)})_+^2\le
\|u_\tau(t)\|_{L^2(\mathbb{R}^d)}^2\le C_2^2(\tau \lceil t/\tau\rceil)^{-2\gamma_2}
+\tfrac{\tilde C_2}{\sqrt{2}}\,\tau \|u_\tau^0\|_2^{2\beta_2},
\]
where $C_2,\tilde C_2,\gamma_2,\beta_2$ are defined in Lemma \ref{lemma:decay2} and
where the right hand side converges, as $\tau\to 0$, to $C_2^2\,t^{-2\gamma_2}$, see Theorem \ref{theop}.  Hence, from \eqref{finalestimate} we obtain
\[
\limsup_{\tau\to0}\|\tilde u_\tau(t)\|_{L^\infty(\mathbb{R}^d)} \le K_{s,d} \,t^{-2\gamma_2\,\frac{(q-1)}{3q-2}-\frac{q}{3q-2}}=K_{s,d}\, t^{-\frac{d}{d+2-2s}},
\]
where 
\[
K_{s,d}:=2^{(2q-1)/(q-1)}S_{d,1-s}^{2q/(3q-2)}C_2^{2(q-1)/(3q-2)},
\]
%\[
%\|v_\tau(t)\|_{L^\infty(\mathbb{R}^d)} \le 2^{q/(q-1)}C_{s,d}^{2q/(3q-2)}\left(C_2(\alpha_2-1)(\tau \lceil t/\tau\rceil)^{-\frac{1}{\alpha_2-1}}
%+\frac\tau2 \|u_\tau(0)\|_2^{2\alpha_2}\right)^{(q-1)/(3q-2)} \, t^{-q/(3q-2)}
%\]
and where we used $2\gamma_2=d/(d+2-2s)$ and $q=d/(d-2+2s)$ to compute the exponent of $t$.
% $$\frac{q-1}{(\alpha_2-1)(3q-2)}+\frac{q}{3q-2}=\frac{d}{d+2-2s}.$$
%therefore the right hand side above converges, as $\tau\to 0$, to the value
%$ K_{s,d} \,t^{-d/(d+2-2s)}$,

By \eqref{Cauchy} with $m=j(\tau)$ we have
\begin{equation}\label{Cauchy2}
	W(u_{\tau}(t),\tilde u_\tau(t))\le \sqrt{2t\FF_s(u_0)}\sum_{j=j(\tau)+1}^{+\infty}2^{-j/2}.
\end{equation}
Since $j(\tau)\to+\infty$ as $\tau\to 0$,  by \eqref{Cauchy2} it follows that along a sequence $\tau_n$ given by Lemma \ref{lemma:convergence2} 
we have that $\{\tilde u_{\tau_n}(t)\}_{n\in\N}$ is tight and converges to the same limit point $u(t)$ of $\{u_{\tau_n}(t)\}_{n\in\N}$.

By lower semicontinuity we conclude that
\[
\|u(t)\|_{L^\infty(\mathbb{R}^d)}\le
K_{s,d}\, t^{-d/(d+2-2s)}.
%\lim_{\tau\to 0} C_2(\alpha_2-1)(\tau \lceil t/\tau\rceil)^{-\frac{1}{\alpha_2-1}}
%+\frac\tau2 \|u_\tau(0)\|_2^{2\alpha_2}
%= C_2(\alpha_2-1)t^{-\frac{1}{\alpha_2-1}}.
\]
The result is achieved with $C_\infty =K_{s,d}$.

{\RRR
{\textbf{{The case \boldmath$d=1$ and \boldmath$0<s<1/2$}}}. The argument is analogous to the previous one for $d\ge 2$, we shall only mention the main differences.   Instead of defining $q=d/(d-2+2s)$, we fix   $r\in(0,1/2)$ and
we let $q:=1/(1-2r)$. 
We define $\theta:=r/(1-s)$, and we change  the definition of $M_{\tau}(t)$ by letting
$$M_\tau(t):=2^{q/(q-1)}S_{1,r}^{2q/(2q-2+q\theta)}A_{j(\tau)}^{(q-1)/(2q-2+q\theta)} \, t^{-q\theta/(2q-2+q\theta)}.$$
Using \eqref{eq:r_interpol} instead of \eqref{fractionalembedding}, the analogous of \eqref{eq:Agei} is
\begin{equation}
\label{eq:Agei2}\begin{aligned}
A_j&\le \left(\frac{2^{j}}{M_{\tau}(t)}\right)^{2q-2}\int_{\mathbb{R}^d}(\tilde u_{\tau,j}(x)-M_{\tau,{j-1}})_+^{2q}\,\d x \\&\le
 \left(\frac{2^{j}}{M_{\tau}(t)}\right)^{2q-2}S_{1,r}^{2q}\,
 \,\|(\tilde u_{\tau,j}-M_{\tau,{j-1}})_+\|^{2q(1-\theta)}_{L^{2}(\mathbb{R}^d)}
 \|(\tilde u_{\tau,j}-M_{\tau,{j-1}})_+\|^{2q\theta}_{\dot H^{1-s}(\mathbb{R}^d)}.
\end{aligned}\end{equation}
Moreover by \eqref{FIM} we have 
\begin{equation}\label{eq:mon}
\|(\tilde u_{\tau,j}-M_{\tau,{j-1}})_+\|^{2}_{L^{2}(\mathbb{R}^d)}\leq \|(\tilde u_{\tau,j-1}-M_{\tau,{j-1}})_+\|^{2}_{L^{2}(\mathbb{R}^d)} = A_{j-1}.
\end{equation}
Using \eqref{positiveflow} and \eqref{eq:mon} in \eqref{eq:Agei2} we obtain the analogous of \eqref{eq:Aj1}
\begin{equation}
\label{eq:Agei3}
A_j\le \left(\frac{2^{j}}{M_{\tau}(t)}\right)^{2q-2}S_{1,r}^{2q}\,  \left(\frac{2^{j}}{tM_{\tau}(t)}\right)^{q\theta} A_{j-1}^q.
\end{equation}
Then we can apply the recursion formula with the choice of 
$Q={S_{1,r}^{2q}}{t^{-q\theta}M_\tau(t)^{2-2q-q\theta}}$ and $R=2^{2q-2+q\theta}$
and we obtain, recalling the choice of $M_\tau(t)$,
\[
A_j\le2^{-(j-j(\tau))(2q-2+q\theta)/(q-1)}\: A_{j(\tau)}.
\]
The rest  of the proof carries over along the line of the case $d\ge 2$.
}
\end{proof}

{\RRR
\begin{proofadmain} 
We collect all the results that give the proof of the main Theorem. 
Point i) follows from Proposition \ref{prop:existenceMM}.  Points ii) and iii) follow from 
Theorem \ref{th:convergence1}, Lemma \ref{lemma:convergence2} and Theorem \ref{th:weak_form}. 
Theorem \ref{th:energy} yields point iv). Point v) is a consequence of Theorem \ref{theop} and Theorem \ref{th:infinity} 
for the case $p <+\infty$ and the case $p=+\infty$, respectively. Finally, point vi) follows from Lemma 
 \ref{lemma:ed} and Lemma \ref{lemma:decay1} by letting $\tau\to 0$ and taking into account the lower semicontinuity 
 of $\mathcal{H}$ and of the $L^p$ norms with respect to the narrow convergence. This gives the result 
 for $p<+\infty$. The case $p=+\infty$ follows by passing to the limit as $p\to +\infty$ in the inequality
$ \| u(t)\|_{L^p(\mathbb{R}^d)} \le  \| u_0\|_{L^p(\mathbb{R}^d)}$. 
\end{proofadmain}

\begin{remark}\label{finalremark}\rm  
If we consider positive measure data with mass $M>0$, according to Remark \ref{firstremark},
%We may generalize Theorem \ref{th:main} to positive measure data in $\dot H^{-s}(\mathbb{R}^d)$, with finite second moment and mass $M>0$. 
%In such case 
the constant $C_p$ in point v) has to be multiplied by $M^{\ell_p}$, where $\ell_p$ is given therein. This scaling is deduced from  Lemma \ref{lemma:decay2} if $p<+\infty$, when making use of \eqref{GNS} and \eqref{extrasobolev1} for obtaining \eqref{E2}.
%Therefore  we obtain a constant $C_p$ in \eqref{discreteRELp} and in Theorem \ref{theop} that scales with $M$ according to Remark \ref{firstremark}.
We similarly obtain the value of $\ell_\infty$, since   the constant $C_\infty$ in Theorem \ref{th:infinity} depends on the mass only through $C_2$.
 \end{remark}
 
 }

\section{The limit for $s\to 0$.}
In this last section we are interested in the {\RRR asymptotic analysis} when $s\to 0$.
We start by proving the following lemma which identifies the limit
 of the sequence of solutions $u_s$ of the equation in \eqref{equation}
as $s\to 0$ with the solutions  
of the porous medium equation \eqref{s=0}.
%\[
%\partial_t u - \frac{1}{2}\Delta u^2 = 0.
%\]

\begin{lemma}
Let $u_0\in L^2(\R^d)$ and $\{u_0^s\}_{s\in(0,1)}$ be a family of initial data such that $u_0^s\in D(\FF_s)$, $u_0^s$ converges narrowly to $u_0$ as
$s\to 0$,  $\sup_{s\in(0,1)} \int_{\R^d} |x|^2 \, \d u_0^s(x) <+\infty$  and 
 $\lim_{s\to 0}\FF_s(u_0^s)=\FF_0(u_0)$.
We denote by $u^s$ a solution of problem \eqref{equation} with initial datum $u_0^s$ given by {\rm Theorem \ref{th:main}}.

If $\{s_n \}_{n\in\N}\subset (0,1)$ is a vanishing sequence, then there exist a curve\\
$u \in AC^2([0,+\infty);(\PP_2(\R^d),W))$
and a subsequence (not relabeled) $\{s_{n}\}$ such that
\begin{equation}\label{nc}
	u^{s_{n}}(t)\to u(t) \quad\mbox{narrowly  as $n\to \infty$ for every $t\geq 0$.}
\end{equation}
Furthermore,  for every $T_0, T$ such that $T>T_0>0$ we have
\begin{equation}\label{sc}
	u^{s_{n}}\to u \quad \mbox{ strongly in $L^2((T_0,T);L^2_{loc}(\mathbb{R}^d))$ \quad as $n\to \infty$,}
\end{equation}
{\RRR and, setting $v^{s_{n}}= K_{s_{n}}\ast u^{s_{n}}$, we have}
\begin{equation}\label{sc1}
	\nabla {\RRR v}^{s_{n}}\to \nabla {u} \quad \mbox{ weakly in $L^2((T_0,T);L^2(\mathbb{R}^d))$ \quad as $n\to \infty$}.
\end{equation}
Moreover, the curve $u$ is a solution of the porous medium equation \eqref{s=0} in the following sense:
\[
\int_0^{+\infty}\int_{\R^d} (\partial_t\varphi - \nabla\varphi\cdot \nabla u) u \,\dd x \,\dd t=0, \quad
\text{for all }\varphi\in C^\infty_c((0,+\infty)\times\R^d)
\]
and the energy dissipation inequality holds
\begin{equation}\label{EI0}
\FF_0(u(T))+\int_0^T\int_{\mathbb{R}^d}|\nabla u(t)|^2u(t)\,\d x\,\d t \le \FF_0(u_0), \qquad \forall \;T>0.
\end{equation}
\end{lemma}
\begin{proof}
Since $\lim_{s\to 0} \FF_s(u^s_0) = \FF_0(u_0)$ we fix $s_0\in(0,1)$ such that $\FF_s(u^s_0) \leq \FF_0(u_0)+1/2$ for any $s\in (0,s_0)$.
Denoting by $|(u^s)'|(t)$ the Wasserstein metric derivative of the curve $t\mapsto u^s(t)$,
by \eqref{md} it holds
\begin{equation}\label{boundmd}
  \int_0^{+\infty}|(u^s)'|^2(r)\,\d r \leq 2 \FF_s(u^s_0) \leq 2\FF_0(u_0)+1 . 
\end{equation}

We have tightness and equicontinuity of the family $\{u^s\}_{s\in (0,s_0)}$.
Indeed, fixing $T>0$, the estimate
$$\begin{aligned}
W^2(u^s(t),\delta_0) &\leq 2W^2(u^s(t),u^s_0)+2W^2(u^s_0,\delta_0)
\leq 2t\!\int_0^t |(u^s)'|^2(r)\,\d r +  \!2\!\int|x|^2u^s_0(x)\,\d x \\
&\leq 2T (2\FF_0(u_0)+1)+  2\sup_{s\in (0,s_0)}\int|x|^2u^s_0(x)\,\d x ,
\end{aligned}$$
implies that the set $\{u^s(t): s\in (0,s_0), t\in [0,T] \}$ is tight and consequently, by Prokhorov Theorem, narrowly compact.

By \eqref{boundmd} there exists $m\in L^2(0,+\infty)$ such that the sequence 
$\{|(u^{s_n})'|\}$ converges to $m$ (up to subsequences) weakly in $L^2(0,+\infty)$.
Then, for every $t_1,t_2 \in [0,+\infty)$, $t_1<t_2$, it holds
\begin{equation}\label{equicont2}
	\limsup_{n\to \infty}W(u^{s_n}(t_2),u^{s_n}(t_1))\leq \lim_{n\to\infty}\int_{t_1}^{t_2} |(u^{s_n})'|(r)\,\d r= \int_{t_1}^{t_2} m(r)\,\d r
\end{equation}
and the equicontinuity is proved.
By the compactness argument of \cite[Proposition~3.3.1]{AGS} we obtain the existence of a continuous limit curve $u$
such that \eqref{nc} holds. 
%{\color{blue}
In particular, since for $t>0$ $u^s(t)$
is absolutely continuous with respect to the Lebesgue measure, \eqref{nc} translates (for $t>0$) into
\begin{equation}
\label{eq:nc2}
\int_{\R^d}u^{s_n}(t,x)\phi(x)\d x \to \int_{\R^d}u(t,x)\phi(x)\d x,\,\,\,\,\,\forall t>0 \,\,\,\,\forall \phi\in C_b(\R^d).
\end{equation}
%}
Passing to the limit in \eqref{equicont2} we obtain 
\begin{equation*}\label{equicont3}
	W(u(t_2),u(t_1))\leq \int_{t_1}^{t_2} m(r)\,\d r,  \qquad \forall\, t_1,t_2 \in [0,+\infty),\quad t_1<t_2,
\end{equation*}
and $u \in AC^2([0,+\infty);(\PP_2(\R^d),W))$.

We fix $\sigma>0$ such that $\sigma<\min\{s_0,1/2\}$.
For $s\in (0,\sigma]$, 
%by Proposition \ref{}
 %$2\FF_s(u)\geq \|u\|^2_{H^{-s}(\R^d)}\geq \|u\|^2_{H^{-\sigma}(\R^d)}$ and consequently 
the energy inequality \eqref{EDI} yields
\begin{equation}\label{negbound}
	 \|u^s(t)\|^2_{H^{-\sigma}(\R^d)}\leq 2\FF_0(u_0)+1, \qquad \forall s\in (0,\sigma], \quad \forall t\in [0,+\infty).
\end{equation}

We fix a compact $K\subset\R^d$ and a compactly supported smooth cutoff function $\phi:\R^d\to[0,1]$ such that $\phi=1$ on $K$.
By interpolation we have
$$ \begin{aligned}
\|u^s(t)-u(t)\|_{L^2(K)} &\leq \|\phi u^s(t)-\phi u(t)\|_{L^2(\R^d)}  \\
	&\leq  \|\phi u^s(t)-\phi u(t)\|^{1/2}_{H^{-1/2}(\R^d)}  \|\phi u^s(t)-\phi u(t)\|^{1/2}_{H^{1/2}(\R^d)} \\
	& \leq  C \|\phi u^s(t)-\phi u(t)\|^{1/2}_{H^{-1/2}(\R^d)}  \| u^s(t)- u(t)\|^{1/2}_{H^{1/2}(\R^d)}.
\end{aligned}$$
By \eqref{negbound}, \eqref{eq:nc2} and the compact embedding of Sobolev spaces it follows that (up to subsequences)
$\lim_{n\to+\infty} \|\phi u^{s_n}(t)-\phi u(t)\|^{1/2}_{H^{-1/2}(\R^d)} =0$.

We fix $T_0>0$ and $T>T_0$. 
By \eqref{L2estimate}, \eqref{RI1} and \eqref{boundentropy}, for $s\leq 1/2$ we have
\begin{equation*}
\begin{aligned}
	&\int_{T_0}^T\|u^s(t)\|_{ H^{1/2}(\mathbb{R}^d)}^2\,dt \leq \int_{T_0}^T\|u^s(t)\|_{ H^{1-s}(\mathbb{R}^d)}^2\,dt \\
	& \le \Big(1+C_0T_0^{-\gamma_0} +T\FF_s(u^s_0)+\int_{\R^d}|x|^2u^s_0(x)\,\d x \Big) \\
	&+ \Big( 2\FF_s(u^s_0) (T-T_0)  \Big)^{1-s} \Big(C_0T_0^{-\gamma_0} + c\Big(1 +T\FF_s(u^s_0)+\int_{\R^d}|x|^2\,\d u^s_0(x) \Big)\Big)^{s},
\end{aligned}
\end{equation*}
where the dependence of the constants $C_0$ and $\gamma_0$ on $s$ is stated in Lemma \ref{lemma:decay2}.
Since $C_0$ is bounded with respect to $s$, it follows that
\begin{equation*}
	\sup_{n\in\N}\int_{T_0}^T\|u^{s_n}(t)\|_{ H^{1/2}(\mathbb{R}^d)}^2\,dt <+\infty,  
	\qquad \int_{T_0}^T\|u(t)\|_{ H^{1/2}(\mathbb{R}^d)}^2\,dt <+\infty.
\end{equation*}
By the previous estimates and dominated convergence theorem we obtain \eqref{sc}.

Analogously from \eqref{vestimate} we obtain
\begin{equation*}\begin{aligned}
	&\int_{T_0}^T\|u^s(t)\|_{ \dot H^{1-2s}(\mathbb{R}^d)}^2\,dt
	\\&\qquad\le (2\FF_s(u_0^s)(T-T_0))^{s} \Big(C_0T_0^{-\gamma_0} +c\big(T\FF_s(u^s_0)+\int_{\R^d}|x|^2u^s_0(x)\,\d x\big)\Big)^{1-s}.
	\end{aligned}
\end{equation*}
%{\color{blue}
Since $ \|\nabla v^s(t)\|_{L^{2}(\R^d)} =  \|u^s(t) \|_{\dot H^{1-2s}(\R^d)}$, taking into account that $C_0$ is bounded as $s\to 0$,
from the previous estimate it follows that
$\{\nabla v^s\}_{s\in(0,\sigma)}$ is weakly compact in $L^2((T_0,T);L^2({\R}^d))$. 
Moreover $\nabla v_{s_n}$ converges to  $\nabla u$ in the sense of distributions in $\R^d\times(T_0,T)$.
Indeed for $\varphi\in C^\infty_c(\R^d\times(T_0,T);\R^d)$, denoting by $\varphi_t$ the function $x\mapsto \varphi(x,t)$, by Plancherel's Theorem 
we have 
$$
(2\pi)^d\int_{T_0}^T\int_{\R^d} \nabla v^{s_n} \cdot \varphi \,\d x\, \d t 
= -i\int_{T_0}^T\int_{\R^d} |\xi|^{-2s_n} \widehat{u^{s_n}(t)}(-\xi) \xi \cdot \widehat\varphi_t(\xi) \,\d\xi\, \d t .
$$
Since  $| |\xi|^{-2s_n} \widehat{u^{s_n}(t)}(-\xi) \xi \cdot \widehat\varphi_t(\xi)| 
\leq \max\{1, |\xi|\} | \widehat\varphi_t(\xi)| $ and $\widehat\varphi_t \in \SS(\R^d)$ for every $t\in(T_0,T)$, by \eqref{nc} and
Lebesgue dominated convergence the right hand side of the above formula converges to
$$
 -i\int_{T_0}^T\int_{\R^d}  \widehat{u(t)}(-\xi) \xi \cdot \widehat\varphi_t(\xi) \,\d\xi\, \d t
= (2\pi)^d\int_{T_0}^T\int_{\R^d} \nabla u \cdot \varphi \,\d x\, \d t .
$$
For the stated compactness in  $L^2((T_0,T);L^2({\R}^d))$ we obtain \eqref{sc1}.
%}

As a result, we can easily pass to the limit in the weak formulation of the equation. 
Concerning the limit procedure in the energy inequality,  
we observe that by \eqref{nc}
and Fatou's lemma we obtain
$$
\liminf_{s\to 0}\FF_s(u^s(t))\geq \FF_0(u(t)).
$$
Moreover by Lemma \ref{lemma:VF} and the stated weak convergence we obtain
$$
\liminf_{s\to 0}\int_0^T\int_{\mathbb{R}^d}|\nabla v^s(t)|^2u^s(t)\,\d x\,\d t\geq \int_0^T\int_{\mathbb{R}^d}|\nabla u(t)|^2u(t)\,\d x\,\d t,
$$
and we conclude.
\end{proof}

%\noindent\emph{Proof of {\rm Theorem \ref{th:convPM}}}.
  \begin{proofad} The proof follows by the previous Lemma and 
the uniqueness of the solution of equation \eqref{pmequation} with initial datum in $L^2(\R^d)$ satisfying the energy inequality
{\RRR (see \cite[Theorem 11.2.5]{AGS}, which also shows that this unique solution satisfies all the properties of \cite[Theorem 11.2.1]{AGS}, in particular the energy identity). }
%Moreover by the geodesic convexity of $\FF_0$ the inequality \eqref{EI0} is an equality.
\end{proofad}

\subsection*{Acknowledgements} 
The authors would like to thank J. A. Carrillo for useful conversations on the topic of this paper. 
{\RRR The authors acknowledge the referees for their careful reading of the paper and for their suggestions.}
E.M. was partially supported by the FWF project M1733-N20.
S.L. and A.S. were  partially supported  by  a  MIUR-PRIN 2010-2011 grant for the project Calculus of Variations.
{\RRR A.S. gratefully acknowledges the financial support of the FP7-IDEAS-ERC-StG \#256872 (EntroPhase).}
 The authors are member of the GNAMPA group of the Istituto Nazionale di Alta Matematica (INdAM).

\subsection*{Conflict of interest}
The authors declare that they have no conflict of interest.

\end{document}